%% file: Paper.tex
\documentclass[final,onefignum,onetabnum,onealgnum,onethmnum,oneeqnum]{siamonline220329}
\input{Shared}

\begin{document}

\maketitle

\begin{abstract}
  Conventional inversion of the discrete Fourier transform (DFT) requires all DFT coefficients to be known. When the DFT coefficients of a rasterized image (represented as a matrix) are known only within a pass band, the original matrix cannot be uniquely recovered. In many cases of practical importance, the matrix is binary and its elements can be reduced to either 0 or 1. This is the case, for example, for the commonly used QR codes. The {\it a priori} information that the matrix is binary can compensate for the missing high-frequency DFT coefficients and restore uniqueness of image recovery. This paper addresses, both theoretically and numerically, the problem of recovery of blurred images without any known structure whose high-frequency DFT coefficients have been irreversibly lost by utilizing the binarity constraint. We investigate theoretically the smallest band limit for which unique recovery of a generic binary matrix is still possible. Uniqueness results are proved for images of sizes $N_1 \times N_2$, $N_1 \times N_1$, and $N_1^\alpha\times N_1^\alpha$, where $N_1 \neq N_2$ are prime numbers and $\alpha>1$ an integer. Inversion algorithms are proposed for recovering the matrix from its band-limited (blurred) version. The algorithms combine integer linear programming methods with lattice basis reduction techniques and significantly outperform naive implementations. The algorithm efficiently and reliably reconstructs severely blurred $29 \times 29$ binary matrices with only $11\times 11 = 121$ DFT coefficients.\\
Published in  \href{https://doi.org/10.1137/22M1540442}{{\em SIAM Journal on Imaging Sciences} {\bf 16}, 1338-1369 (2023)}\\
doi: \href{https://doi.org/10.1137/22M1540442}{10.1137/22M1540442}
\end{abstract}

\begin{keywords}
Two-dimensional discrete Fourier transform, recovery of binary matrices, inversion, deblurring
\end{keywords}

\begin{MSCcodes}
94A08, 68U10, 65T50
\end{MSCcodes}

\section{Introduction}
\label{sec:intro}

The paper address the problem of reconstruction of binary images from limited sets of discrete Fourier transform (DFT) coefficients. We are interested in exact pixel-by-pixel reconstruction of general images without any structure or known properties, i.e., under the conditions when the methods based on machine learning are not expected to be efficient. Images whose DFT coefficients are lost outside of a given pass band are blurred and therefore the problem we are addressing is that of de-blurring. A typical application is de-blurring of QR codes or similar rasterized images in which only two colors are present. Forms such as Data Matrix codes and QR codes are used in applications ranging from industrial tracking to advertising~\cite{gao_2007_1}. If the stored information is lost due to a corrupted signal at high frequencies, the results of this paper allow one to recover the original code. Therefore, the main advance reported below is the ability to reconstruct not very large but seemingly random binary images. The paper builds upon our previous results for the one-dimensional case~\cite{levinson_2021_1}, which were, in turn, related to the work of Tao~\cite{tao_2005_1}, Tropp~\cite{tropp_2008_1}, and the recent work of Pei and Chang~\cite{pei_2022_1}.  
 
Images are often blurred as a result of low-pass filtering, either due to physical limitations of the image acquisition process~\cite{born_book_1999, maznev_2017_1}, or due to application of various filters for image denoising and compression~\cite{gilbert_2014_1, hu_2014_1}. In either case, DFT coefficients of the blurred image outside of the pass band are below the noise level and, for practical purposes, lost. If no additional information is available, it is, in principle, impossible to recover the image precisely. However, if it is known {\em a priori} that the original image is binary (contains only two known values), and enough DFT coefficients are known with sufficient precision, we can utilize the binarity constraint to reconstruct all pixels precisely. This is demonstrated below both theoretically in the form of uniqueness theorems and numerically for severely blurred QR codes with the size of up to $29\times 29$.

Of course, once an image is recovered, we can also compute all of its DFT coefficients, including those that were not known beforehand. We say that, by retrieving the DFT coefficients located outside of the original pass band, we increase the image resolution. If the loss of resolution occurred due to physical limitations of the image acquisition process (such as exponential decay of evanescent waves), and we have recovered the DFT coefficients outside of the physically-imposed pass band, we say that we have achieved the effect of {\em super-resolution} -- that is, we have resolved computationally the details that are not visible directly under the experimental conditions. 

In image de-blurring applications, {\em a priori} information unrelated to the missing DFT coefficients is often available. In such cases, powerful techniques can be developed to achieve recovery of the exact image. Feasibility of achieving super-resolution with meaningful prior information has been demonstrated in many works~\cite{moitra_2015_1, nasrollahi_2014_1, romano_2017_1}. A well-studied example is the case of sparse images, which contain relatively few nonzero pixels. It was shown that the knowledge that the original image is sparse allows for stable recovery with severely under-sampled measurements~\cite{donoho_1989_1, elad_2002_1, candes_2006_1}. Corresponding fast reconstruction algorithms have been extensively developed~\cite{tropp_2007_1, blumensath_2009_1, blumensath_2012_1, blumensath_2013_1}. The sparsity constraint can be independent of the Fourier bases, but there exist many relevant results specific to the Fourier coefficients, including those applicable to random~\cite{rauhut_2007_1, rudelson_2006_1} and deterministic measurements~\cite{bailey_2012_1}.  In particular, sparse fast Fourier transform techniques~\cite{plonka_2018_1, plonka_2021_1} are used to quickly recover sparse vectors, that may or may not have additional known structure.  In these problems, however, sampling of high frequency DFT coefficients is required, which are outside the typical pass band considered in this paper.  Additional techniques for achieving super-resolution (non-sparsity regularization frameworks) have also been developed, including nonlinear interpolation~\cite{rajan_2001_1, jiji_2006_1}, Laplacian~\cite{lagendijk_1990_1, liu_2014_1} and total variation~\cite{beck_2009_1, rudin_1992_1} regularizations. 

However, the above techniques rely on assumptions about the images, which limit generality of application and which we wish to avoid in this paper. Instead, we utilize a different, yet still a fundamental constraint. Namely, we consider the case when each pixel of the image can take only two different, {\em a priori} known values. As was shown in our previous work~\cite{levinson_2021_1}, the problem can be reduced by a simple transformation to that of recovering an image whose pixel values can be either 0 or 1. We say that such images are binary. We will use no additional assumptions on the spatial distribution of zeros and ones, and will be interested in recovering the original image precisely from a limited set of DFT coefficients. Note that, while there exists some overlap between the conditions of binarity and sparsity, a binary image can contain substantially more nonzero entries (roughly half of the total) than a typical sparse image. In such cases, sparsity-based recovery methods are not efficient. 

Binary images and matrices have been extensively studied in the literature, motivated by applications to imaging~\cite{dizenzo_1996_1, marchand_book_1999, ren_2002_1} and combinatorics~\cite{fulkerson_1960_1, brualdi_1980_1, brualdi_1986_1, snijders_1991_1}. Recovering binary images from incomplete data is closely related to the problem of discrete tomography~\cite{gardner_1997_1, gardner_1999_1, hajdu_2001_1, herman_book_2012}. Here one tries to reconstruct a binary image from families of parallel line integrals (projections) with a small number of specified angles. This mathematical technique has applications to medical imaging~\cite{herman_2003_1}. In this paper, we start with DFT coefficients and show that the knowledge of some small sub-sets of such coefficients is similar to the knowledge of some selected projections, except that the line integrals of this paper are periodic in nature, unlike those that arise in discrete tomography. We note that Fourier transforms~\cite{wang_1998_1, yagle_1998_1, yagle_2001_1} as well as specific periodic constraints~\cite{del_2002_1} have been previously used in discrete tomography. However, additional prior information is typically used in these applications (such as connectedness) to find a binary image that is physiologically realizable. We do not apply such constraints and consider a more general problem.

The main theoretical question addressed in this paper is the following: how many DFT coefficients are needed to uniquely determine a binary matrix? We assume that the measurements are deterministic and available within a low spatial frequency region (pass band) as defined more precisely below. We will also be interested in recovering the image numerically. However, even if uniqueness is guaranteed, recovery of the exact binary matrix without any known structure is  an NP-hard problem~\cite{gardner_1999_1, karp_1972_1}. In the most combinatorially challenging regime wherein roughly half of the entries are ones and the rest are zeros, the binary matrix is not sparse. We therefore cannot use the conventional avenues for improving the computational efficiency of recovery. Instead, we solve the inversion problem using integer linear programming and lattice basis reduction techniques. While naive implementations of integer linear programming quickly hit computational roadblocks and are limited to matrices with $\lesssim 50$ entries (i.e., of the size $7 \times 7$ or less), we have developed algorithms specifically tailored to the problem at hand. The largest image size for which the algorithm was successfully tested is $29 \times 29$ with $841$ pixels. We note that our algorithm allows to recover uniquely any of the $2^{841}$ distinct binary images of this size using only $11 \times 11=121$ DFT coefficients.

We use typewriter-style straight letters to denote matrices (as in ${\tt X}$) and vectors (as in ${\tt x}$). Elements of these structures, as well as other scalar quantities, are denoted by italic letters as in $X_{nm}$ or $x_n$. Fourier transforms are denoted by overhead tilde. For example, $\tilde{\tt X}$ is a matrix of complex DFT coefficients of ${\tt X}$ and $\tilde{X}_{kl}$ is a particular element of $\tilde{\tt X}$. The greatest common divisor of two integers $n$ and $m$ is denoted by $\gcd(n,m)$, and we let $\mathbb{Z}_N$ denote the ring of integers modulo $N$.  

\section{Theoretical background}
\label{sec:theo}

\subsection{Statement of the inverse problem}
\label{sec:theo.IP}

Let ${\tt X}$ be an $N_1 \times N_2$ matrix, and assume that its entries $X_{mn}$ can take only two values, either 0 or 1. The DFT of ${\tt X}$ is given by
\begin{align}
\label{DFT_dir}
\tilde{X}_{kl} = \sum_{m=1}^{N_1}\sum_{n=1}^{N_2}X_{mn} e^{2\pi\cu \left( mk/N_1 + nl/N_2 \right)} \ .
\end{align}
The DFT coefficients $\tilde{X}_{kl}$ are periodic in each index, so that $\tilde{X}_{k l} = \tilde{X}_{k+N_1, l+N_2}$. Since we will mainly be considering the cases when both $N_1$ and $N_2$ are odd, it is sufficient to restrict the indexes $k,l$ to the symmetric intervals
\begin{align}
\label{eq:oddM}
-M_1\le k \le M_1 \ , \ \ -M_2 \le l \le M_2 \ , \ \ \mbox{where} \ M_1=(N_1-1)/2 \ ,  \ \ M_2=(N_2-1)/2 \ .
\end{align}
Then $\tilde{\tt X}$ is the $N_1 \times N_2$ matrix of DFT coefficients $\tilde{X}_{kl}$ with the indexes restricted by \cref{eq:oddM}. The inverse DFT is defined as
\begin{align}
\label{DFT_inv_2D}
X_{mn} = \frac{1}{N_1 N_2} \sum_{k=-M_1}^{M_1} \sum_{l=-M_2}^{M_2} \tilde{X}_{kl} e^{-2\pi\cu \left( mk/N_1 + nl/N_2 \right)} \ ,
\end{align}
which allows for reconstruction of the original matrix ${\tt X}$ from the knowledge of $\tilde{\tt X}$.   Generically, if some of the elements of $\tilde{\tt X}$ are not known, none of the elements of ${\tt X}$ can be reconstructed uniquely. Indeed, it can be seen from \cref{DFT_inv_2D} that changing only one element of $\tilde{\tt X}$ changes all elements of ${\tt X}$. 

However, with the additional constraint that the elements of ${\tt X}$ are binary,  we can hope to achieve unique inversion from only partial knowledge of $\tilde{\tt X}$. We will therefore address the following question: is it possible to reconstruct ${\tt X}$ precisely from the knowledge of only a proper subset of its DFT coefficients? The precise problem definition is as follows.

\begin{definition}
\label{def:IP}
\textcolor{magenta}{\rm We use the acronym {\rm IP}$(N_1,N_2,L_1,L_2)$ to denote the inverse problem of reconstructing a generic binary matrix ${\tt X}$ of known dimension $N_1\times N_2$ from the set of its DFT coefficients $\tilde{X}_{kl}$ with indexes restricted by
\begin{align}
\label{band_def}
|k| \le L_1\le M_1 \ , \ \ |l| \le L_2\le M_2 \ .
\end{align}
We refer to two binary matrices ${\tt X}$ and ${\tt Y}$ as being $(L_1, L_2)$-indistinguishable if they have the same DFT coefficients within the band \cref{band_def}. In the case $L_1 = L_2 = L$, we use the shorthand ``$L$-indistinguishable''.}
\end{definition}

Note that, since ${\tt X}$ is real, we have $\tilde{X}_{-k,-l} = \tilde{X}^*_{k,l}$. Consequently, there are $(L_1+1)(L_2+1) + L_1 L_2$ independent complex coefficients in the band \cref{band_def}, ignoring the pairs that are known conjugates of each other.

The DFT coefficient that is always accessible in this setup is the popcount, $S \equiv \tilde{X}_{00}$, which gives the total number of ones in ${\tt X}$. We thus assume that the value of $S$ is always known. In general, the problem of recovering a binary matrix ${\tt X}$ from a limited set of DFT coefficients is most challenging when $S \sim (N_1 N_2)/2$. This is so because the total number of binary matrices with $S$ nonzero entries is given by ${N_1N_2 \choose S}$.

\subsection{Cyclotomic Integers}
\label{sec:theo:sru}

One key tool that we will use to determine whether a binary matrix is uniquely recoverable from a certain subset of DFT coefficients is analysis of sums of complex exponentials with integer coefficients. If two binary matrices ${\tt X}$ and ${\tt Y}$ have the same $(k,l)$-DFT coefficient, then, by linearity of the DFT, we have $\tilde{Z}_{kl} = 0$, where ${\tt Z} = {\tt X} - {\tt Y}$. Thus, it is useful to know under what conditions a sum of roots of unity can be zero. This problem has been studied extensively. Some relevant results pertaining to the case when the roots of unity are all of the same order are summarized below.  

Consider the $N$-th roots of unity, which are the $N$ solutions to the equation $z^N=1$. These solutions are of the form $e^{2\pi \cu k/N}$, $k=1,\dots, N$. If $\gcd(k,N)=1$, then $e^{2\pi \cu k/N}$ is a primitive root of unity, and it is not a solution to the equation $z^M=1$ for any integer $M<N$. Let $\zeta_N$ be a primitive $N$-th root of unity, and suppose that
\begin{align}
\label{eq:vanish}
\sum_{n=1}^N a_n \left( \zeta_N \right)^n = 0 \ ,
\end{align}
where the coefficients $a_n$ are all integers. The sum appearing on the left-hand side of this expression is known as a {\it cyclotomic integer} -- a linear combination of $N$-th roots of unity with integer coefficients.

First, consider the case when $N$ is a prime number. Since the cyclotomic polynomial $1+x+x^2+\cdot+x^{p-1}$ is irreducible, the equality \cref{eq:vanish} can hold only if $a_n=c$ for all $1\le n\le N$, where $c$ is some constant integer (see proof of Theorem 1 of~\cite{levinson_2021_1}, for example). Thus, an important consequence of irreducibility of the cyclotomic polynomial is that, if a cyclotomic integer of prime order is equal to $0$, then all of its coefficients are the same constant integer.

Such a strong condition does not hold if $N$ is not prime. However, one can still obtain conditions depending on the prime factors of $N$. The main result for integer vanishing sums of roots of unity is given by the following two Lemmas as stated in~\cite{lenstra_1978_1}.

\begin{lemma}
\label{lem:lenstra1}
\textcolor{blue}{\rm Let $M$ be the product of all distinct primes dividing $N$, and let $\zeta_M$ and $\zeta_N$ be primitive $M$-th and $N$-th roots of unity, respectively. Then $\{\left( \zeta_M \right)^m \left( \zeta_N \right)^n \ : \ 1 \le m \le M, \ 1 \le n \le N/M\}$ is the complete set of $N$-th roots of unity. Moreover, for $a_{mn} \in \mathbb{Z}$, the following equation holds
\begin{displaymath}
\sum_{m=1}^{M} \sum_{n=1}^{N/M} a_{mn} \left( \zeta_M \right)^m \left( \zeta_N \right)^n = 0
\end{displaymath}
if and only if
\begin{displaymath}
\sum_{m=1}^M a_{mn} \left( \zeta_M \right)^m = 0 \ \ \ \mbox{\rm for all $n$ such that $1 \le n \le N/M$} \ .
\end{displaymath}}
\end{lemma}

\begin{lemma}
\label{lem:lenstra2}
\textcolor{blue}{\rm Let $N=pM$, where $p$ is prime and does not divide $M$, and let $\zeta_M$ and $\zeta_p$ be primitive $M$-th and $p$-th roots of unity, respectively. Then $\{ \left( \zeta_M \right)^m \left( \zeta_p \right)^n : 1 \le m \le M, \ 1 \le n \le p\}$ is the complete set of $N$-th roots of unity. Then, for $a_{nm} \in  \mathbb{Z}$, the following equality holds
\begin{displaymath}
\sum_{m=1}^{M} \sum_{n=1}^{p} a_{mn} \left( \zeta_M \right)^m \left(\zeta_p \right)^n = 0
\end{displaymath}
if and only if
\begin{align}
\label{eq:2divisors_1}
\sum_{m=1}^M a_{mn} \left( \zeta_M \right)^m = \sum_{m=1}^M a_{m1} \left( \zeta_M \right)^m \ \ \ \mbox{\rm for all $n$ such that $1 < n \le p$} \ .
\end{align}}
\end{lemma}

\Cref{lem:lenstra1} is used to analyze roots of unity of order $N$ where $N$ has at least one prime power as a divisor.  \Cref{lem:lenstra2} provides a tractable condition when $N$ has only two prime divisors. In this case, $M$ in \cref{eq:2divisors_1} is prime; therefore, by subtracting the two sums, we have a vanishing cyclotomic integer as in \cref{eq:vanish}. Thus, we can conclude that, for each fixed $n$, $(a_{mn} - a_{m1})$ is constant for $1 \le m \le M$. If $N$ has more than two prime divisors, it is much harder to analyze \cref{eq:2divisors_1} due to existence of the so-called asymmetrical sums~\cite{conway_1976_1, lam_2000_1}.

Building on these ideas, our previous work~\cite{levinson_2021_1} developed the theory of recovering binary one-dimensional signals from limited sets of DFT coefficients. Results were obtained for vectors of prime length $N$, and of length of the form $N=pq$ where $p$ and $q$ are two (possibly, equal) prime factors. Two-dimensional binary DFT requires a separate analysis, but some results can be generalized from the one-dimensional setting. We therefore briefly summarize the pertinent one-dimensional theory below.

\subsection{Summary of results on binary vectors}
\label{sec:theo.1D}

For vectors ${\tt x}$ of length $N$, the one-dimensional DFT is defined as 
\begin{equation}
\tilde{x}_m = \sum_{n=1}^N x_ne^{2\pi i m/N} \ .
\end{equation}
When ${\tt x}$ is known to be a binary vector of prime length $p$, inversion is unique with the knowledge of the first two DFT coefficients $\tilde{x}_0$ and $\tilde{x}_1$.   This is a consequence of the irreducibility of cyclotomic polynomials (see Theorem 1 of \cite{levinson_2021_1}). For binary vectors of length $pq$ (where, possibly, $p=q$), the results are more subtle. Many such vectors are uniquely recoverable from only their first two DFT coefficients, but some vectors, which have a special structure, are not. The result is stated below as \cref{thm:1D}, which was proved in a rephrased form in~\cite{levinson_2021_1}.

\begin{lemma}
\label{thm:1D}
\textcolor{blue}{\rm Let ${\tt x}$ be a binary vector of length $pq$, where $p$ and $q$ are (not necessarily distinct) prime numbers. Then ${\tt x}$ is not uniquely determined by its DFT coefficients $\tilde{x}_0$ and $\tilde{x}_1$ (that is, there exists a distinct vector ${\tt y} \ne {\tt x}$ with $\tilde{y}_0 = \tilde{x}_0$ and $\tilde{y}_1 = \tilde{x}_1$) if and only if, for $s=p$ or $s=q$, ${\tt x}$ has indexes $a,b\in\{1,\dots,pq\}$ such that the following two conditions hold simultaneously: 
\begin{align}
\label{eq:polygon}
\Big{\{} x_\alpha = 1 \ \ \mbox{\rm for all $\alpha = a\pmod s$} \Big{\}} \ \ \ \mbox{\rm AND} \ \ \ 
\Big{\{} x_\beta  = 0 \ \ \mbox{\rm for all $\beta  = b \pmod s$} \Big{\}} \ .
\end{align}
Moreover, if ${\tt x}$ is not uniquely determined by $\tilde{x}_0$ and $\tilde{x}_1$, then a distinct binary vector ${\tt y}$ is 1-indistinguishable from ${\tt x}$ if and only if ${\tt y}$ satisfies \cref{eq:polygon} for the same $a$ and $b$, except for the permutation $0 \leftrightarrow 1$, that is we write $y_\alpha=0$ and $y_\beta=1$.}
\end{lemma}

\section{Uniqueness results}
\label{sec:uni}

In this section, we state and prove uniqueness results for binary matrices of the size $N_1 \times N_2$. Due to the complexity associated with the asymmetric sums of roots of unity, we assume below that the total number of pixels, $N_1N_2$, has no more than two prime divisors. The cases we cover are not exhaustive, but give a taste for the type of super-resolution one can obtain for binary matrices. 

\subsection{Row- and column-wise popcounts}
\label{sec:theo.rc}

As previously mentioned, the global popcount (the total number of ones in ${\tt X}$) is given by $S = \tilde{X}_{00}$. We also define the row- and column-wise popcounts $r_m$ and $c_n$ as
\begin{align}
\label{rm_def}
r_m = \sum_{n=1}^{N_2}X_{mn} \ , \ \ c_n=\sum_{m=1}^{N_1}X_{mn} \ .
\end{align}
If the dimensions $N_1$ and $N_2$ are both prime, the next two lowest-order DFT coefficients of ${\tt X}$ fix all $r_m$ and $c_n$. For example, the coefficient $\tilde{X}_{10}$ is given by
\begin{align}
\label{X10_DFT}
\tilde{X}_{10} = \sum_{m=1}^{N_1} r_m  e^{2\pi\cu m/N_1} \ .
\end{align}
The right-hand side of \cref{X10_DFT} is a cyclotomic integer -- a sum of powers of a primitive root of unity with integer coefficients. Assuming that the global popcount $S$ and $\tilde{X}_{10}$ are known, all $r_m$'s are also known (as the cyclotomic integers are irreducible). This statement is a slight generalization of the result of ~\cite{levinson_2021_1} where we proved that \cref{X10_DFT} is uniquely invertible for binary $r_m$; here we say that it is uniquely invertible for integer $r_m$. The proof is a trivial extension of the proof given in~\cite{levinson_2021_1}. Similarly, the knowledge of $\tilde{X}_{01}$ fixes all column-wise popcounts $c_n$. Note that this geometric equivalence is only true when $N_1$ and $N_2$ are prime. 

Thus, the knowledge of $\tilde{X}_{00}$, $\tilde{X}_{01}$ and $\tilde{X}_{10}$ is sufficient to recover the global and the row- and column-wise popcounts assuming $N_1$ and $N_2$ are prime. In some special cases, this information defines uniquely the whole binary matrix (a trivial example is when $S=1$). In general, this is clearly false. The problem of determining a binary matrix by its row- and column-wise sums has been extensively studied and solved~\cite{ryser_1957_1, ryser_1960_1}. In particular, two binary matrices ${\tt X}$ and ${\tt Y}$ have the same row- and column-wise sums if they differ by an {\it interchange}, where an interchange is defined by a quadruple $(k,l,m,n)$ such that
\begin{align*}
\begin{bmatrix}
X_{kl} & X_{kn} \\ 
X_{ml} & X_{mn}
\end{bmatrix}
=
\begin{bmatrix}
1 & 0 \\ 
0 & 1
\end{bmatrix}
\ \ , \ \ \ \ 
\begin{bmatrix}
Y_{kl} & Y_{kn} \\ 
Y_{ml} & Y_{mn}
\end{bmatrix}
=
\begin{bmatrix}
0 & 1 \\ 1 & 0 
\end{bmatrix} \ .
\end{align*}
Moreover, any two matrices with equivalent row and column sums can be obtained from one another by a sequence of such interchanges. 

These results imply that, except for some very special cases, uniquely determining a binary matrix from its row- and column-wise popcounts is an impossible task. In what follows, we investigate how many additional DFT coefficients are required to make all binary matrices of a given size uniquely recoverable. Below, we study matrices of dimensions $N_1 \times N_2$ and consider the cases (i) when $N_1$ and $N_2$ are distinct primes,  (ii) square matrices with $N_1=N_2=N$ and prime $N$, and (iii) square matrices with $N=p^\alpha$ where $p$ is prime and $\alpha>1$ is an integer.  

\subsection{Matrices of sizes $N_1\times N_2$ with distinct primes $N_1$ and $N_2$}
\label{sec:uni.rect}

For rectangular matrices with prime dimensions, we can prove our strongest uniqueness result. With the knowledge of just one additional DFT coefficient (in addition to $\tilde{X}_{00}$, $\tilde{X}_{01}$ and $\tilde{X}_{10}$), the binary matrix ${\tt X}$ can be uniquely recovered. In line with our assumption of low frequency coefficients becoming available first, this additional DFT coefficient is $\tilde{X}_{11}$. Note that this is a stronger restriction than the notation IP$(N_1,N_2,1,1)$ conveys, which includes all DFT coefficients in the pass band with $|k|,|l|\le 1$. However, we will show that uniqueness does not require the knowledge of $\tilde{X}_{1,-1}$ or of its equivalent conjugate pair.

\begin{theorem}
\label{thm:2D_rect}
\textcolor{red}{\rm Consider a generic binary matrix ${\tt X}$ of dimension $N_1 \times N_2$, where $N_1$ and $N_2$ are prime and $N_1\neq N_2$. If the four DFT coefficients $\tilde{X}_{00}$, $\tilde{X}_{10}$, $\tilde{X}_{01}$, and $\tilde{X}_{11}$ are known, then the inverse problem of reconstructing ${\tt X}$ is uniquely solvable.}
\end{theorem}

\begin{proof}
Denote the total number of elements as $T=N_1 N_2$. Let ${\tt X}$ and ${\tt Y}$ be two distinct $N_1 \times N_2$ binary matrices. Suppose that $\tilde{X}_{kl} = \tilde{Y}_{kl}$ for $0\le k,l \le 1$. Consider the (1,1)-th DFT coefficient of ${\tt X}$,
\begin{align}
\label{X11_DFT}
\tilde{X}_{11} = \sum_{m=1}^{N_1}\sum_{n=1}^{N_2} X_{mn} e^{2\pi\cu \left(m/N_1 + n/N_2 \right)}
= \sum_{m=1}^{N_1} \sum_{n=1}^{N_2} X_{mn} e^{2\pi\cu \left( m N_2 + n N_1 \right) / T} \ .
\end{align}
As $N_1$ and $N_2$ are distinct primes, $e^{2\pi\cu \left(N_1 + N_2 \right)/T}$ is a primitive root of unity of $T$-th order, with the complete set of $T$-th roots of unity given by 
\begin{align*}
\{ e^{2\pi\cu \left( m N_2 + nN_1 \right) / T } \ : \ 1 \le m \le N_1 \ , \ 1 \le n\le N_2\} \ .
\end{align*}  
These are the roots that appear in \cref{X11_DFT}, suggesting that the sum is the one-dimensional DFT coefficient of some vector ${\tt x}$. Let ${\tt x}$ be the binary vector of length $T$ formed by unrolling the entries of ${\tt X}$ according to
\begin{align}
\label{1D_to_2D_alpha}
x_{\alpha} = X_{mn} \ , \ \ \alpha = mN_2 + n N_1 \pmod T \ .
\end{align}
We can thus rewrite \cref{X11_DFT} as
\begin{align*}
\tilde{X}_{11} = \sum_{n=1}^{T} x_n e^{2\pi\cu (n/T)} = \tilde{x}_1 \ ,
\end{align*}
which is equivalent to the first DFT coefficient of the one-dimensional binary vector ${\tt x}$. Similarly define the binary vector ${\tt y}$ such that $\tilde{Y}_{11} = \tilde{y}_1$. Thus, we have two distinct one-dimensional binary vectors, ${\tt x}$ and ${\tt y}$ of length $T$ each, which agree at their first two DFT coefficients. By \cref{thm:1D}, ${\tt x}$ and ${\tt y}$ must agree at all entries, except on at least one pair of indexes $a,b\in\{1,\dots,T\}$ that satisfy \cref{eq:polygon}. Assuming $s = N_1$ in \cref{thm:1D}, we have $x_\alpha=1$ for all $\alpha=a \pmod{N_1}$. Applying this result to \cref{1D_to_2D_alpha}, there exists a fixed value of $m = m_0$ such that $X_{m_0 n}=1$ for all $1\le n\le N_2$. We can similarly conclude that $Y_{m_0 n}=0$ for all $1\le n\le N_2$. However, $\tilde{X}_{10}$ and $\tilde{Y}_{10}$ fix the row sums of the matrices ${\tt X}$ and ${\tt Y}$. As  $\tilde{X}_{10} = \tilde{Y}_{10}$ by assumption, ${\tt X}$ and ${\tt Y}$ must have the same row sums. We have, in contradiction, already shown that the $m_0$-th row of ${\tt X}$ has a row sum of $N_2$ whereas the same row of ${\tt Y}$ sums to $0$. Identical logic holds for the case when $s=N_2$ in \cref{thm:1D} by, instead, finding a fixed column index that has differing sums for ${\tt X}$ and ${\tt Y}$. This contradicts the assumption that $\tilde{X}_{01} = \tilde{Y}_{01}$. Thus, by \cref{thm:1D}, as ${\tt x}$ and ${\tt y}$ agree on their 1st one-dimensional DFT coefficient, but do not differ at the stated indexes, they must be equal. Hence, by \cref{1D_to_2D_alpha}, ${\tt X} = {\tt Y}$, making the solution to the inverse problem unique.
\end{proof}

While results for binary one-dimensional vectors were used in the proof of \cref{thm:2D_rect}, the conclusion of this theorem is significantly stronger than in the one-dimensional case. Indeed, for vectors of length $T = N_1 N_2$ with $N_1<N_2$ being both prime, one requires $L = N_2$ to guarantee uniqueness by Lemma 2 of~\cite{levinson_2021_1}. In contrast, for matrices of the dimension $N_1 \times N_2$, the required number of DFT coefficients does not increase with $N_1$ or $N_2$ but rather stays fixed at $4$.

\subsection{Square matrices of prime order $N$}
\label{sec:uni.sqp}

While results for binary vectors of length $T = N_1 N_2$ were used in the above proof of \cref{thm:2D_rect} for rectangular matrices, we cannot use the same approach for square matrices. This is so because, for a binary matrix ${\tt X}$ of dimension $N \times N$, the expression for $\tilde{X}_{11}$ no longer involves a complete, non-repeating set of roots of unity as in \cref{X11_DFT}. Instead, we have
\begin{align}
\label{X11_square}
\tilde{X}_{11} = \sum_{m=1}^N \sum_{n=1}^N X_{mn} e^{2\pi \cu (m+n)/N} \ .
\end{align}
The exponential factors in the right-hand side of \cref{X11_square} are the $N$-th roots of unity, and each root appears $N$ times (there are $N^2$ terms in the summation). Albeit different than in the rectangular case, equation \cref{X11_square} contains useful geometric information about the elements of ${\tt X}$, similarly to the coefficients $\tilde{X}_{10}$ and $\tilde{X}_{01}$, which contain information about the number of nonzero entries in each row and column, respectively. To see that this is the case, we rewrite \cref{X11_square} by grouping the roots of unity as \hspace*{6cm}
\begin{wrapfigure}{r}{4.2cm}
\centering
\includegraphics[width=4.0cm]{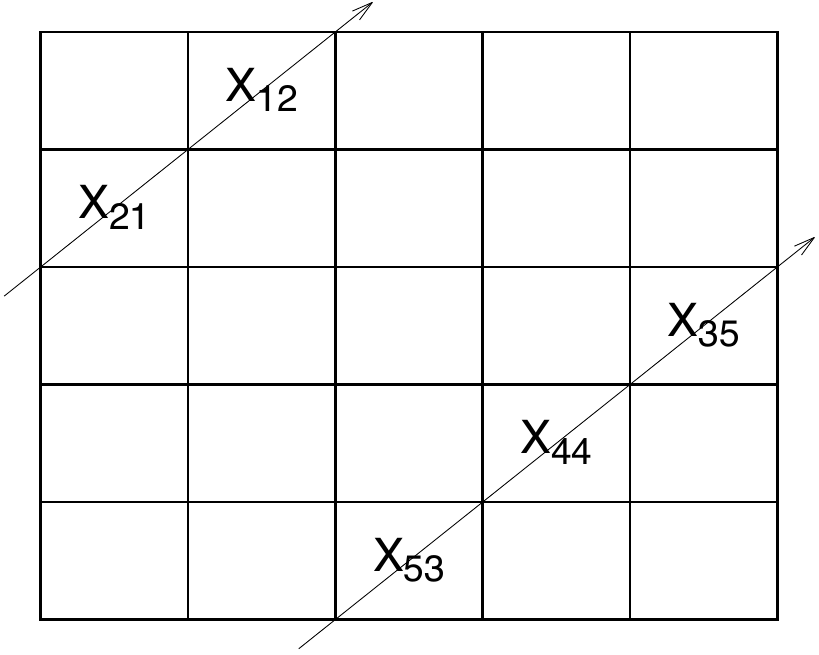}
\caption{\rm \label{fig:p5} Solutions to \cref{eq:line_mod} with $N = 5$ and $j=3$. This line is referred to as having slope -1 as the column index increases by 1 as the row index decreases by 1.}
\end{wrapfigure}
\begin{align}
\tilde{X}_{11}=\sum_{j=1}^N \ \eta_j \ e^{2\pi \cu (j / N)}  \ , \ \ \mbox{where} \ \eta_j = \hspace*{-6mm} \sum_{\substack{m,n=1 \\ m+n=j \, ({\rm mod}\, N)}}^N \hspace*{-6mm} X_{mn} \ .
\end{align}  
Using the fact that the cyclotomic integers are irreducible, we conclude that the knowledge of $\tilde{X}_{11}$ is equivalent to knowing the values of $\eta_j$ for $j=1, \dots, N$. This, in turn, tells us how many ones are in each subset (labeled by $j$) of elements $X_{mn}$ with indexes $m,n$ satisfying the equation 
\begin{align}
\label{eq:line_mod}
m+n = j \pmod N \ . 
\end{align}
For each fixed $j$, the $N$ solutions to \cref{eq:line_mod} lie along a line of the slope $-1$, which may be periodically extended. This is illustrated in \cref{fig:p5}. Thus, the value of $\tilde{X}_{11}$ tells us how many ones are in each line of slope -1. In this sense, $\tilde{X}_{11}$ provides projection information similar to that in $\tilde{X}_{10}$ and $\tilde{X}_{01}$, but along the lines that are neither horizontal nor vertical but have the slope of $-1$.  

It is a straightforward extension to show that $\tilde{X}_{kl}$ contains information equivalent to the projection along a periodic line defined by the equation
\begin{align}
\label{eq:line_mod2}
k n + l m = j \pmod N \ .
\end{align}
We say that the slope of the line defined by \cref{eq:line_mod2} is $-l / k$. Note that the expression \cref{eq:line_mod2} is valid for $k\neq 0$. If $k=0$, $\tilde{X}_{0 l}$ counts the number of nonzero entries along the vertical lines. An immediate consequence of the above observation is that $\tilde{X}_{kl}$ and $\tilde{X}_{k'l'}$ provide the same information if the periodic line classes with $(k,l)$ and $(k',l')$ have the same slope. This happens whenever
\begin{align}
\label{eq:equiv_slope}
(k,l) \sim (k',l') \iff kl' = k'l \pmod N  \ ,
\end{align}
where we have stated the condition as an equivalence relation. Note that, in general, this is a valid equivalence relation whenever at least $k$ and $k'$ (or $l$ and $l'$) are relatively prime to the congruent modulo number, which is always true if $N$ is prime, as we assume here. 

We are thus considering a periodic extension of the standard problem concerning row and column sums of binary matrices considered in~\cite{ryser_1957_1, ryser_1960_1}. Instead of asking when a binary matrix can be uniquely determined by its projections along horizontal and vertical lines, we are interested in how many periodic projections (and in which directions) are sufficient to uniquely recover an $N \times N$ binary matrix. The key idea here is that, while in general we need all $N^2$ DFT coefficients to determine the original matrix (or $(N^2+1)/2$ by symmetry when the matrix is known to be real), in this binary setup, many of the DFT coefficients  contain the same information as another coefficient. For example, it is easy to see that, for any prime $N$, $\tilde{X}_{2,0}$ also gives the individual popcount along each row of ${\tt X}$ and provides no additional information compared to $\tilde{X}_{1,0}$. As another example, let $N=23$; then, according to \cref{eq:equiv_slope}, $\tilde{X}_{7,5}$ provides the same information as $\tilde{X}_{1,6}$. Thus, it is clear that we should not need all $N^2$ DFT coefficients to recover ${\tt X}$ as there are fewer than $N^2$ independent coefficients. The following lemma, originally due to Thue, is the key algebraic result for determining how many coefficients are required for unique recovery.

\begin{lemma}
\label{lemma3}
\textcolor{blue}{Let $N$ be prime and define $L_0 = \lfloor \sqrt{N} \rfloor$. Let $k$ and $l$ be integers such that $ |k|,|l| \le N-1$. Then there exist integers $k'$ and $l'$ with $ |k'|,|l'| \le L_0$ such that $kl'=k'l \pmod N$.}
\end{lemma}

A proof can be found in~\cite{shoup_book_2009}. \Cref{eq:equiv_slope} provides the condition under which two DFT coefficients are dependent. \Cref{lemma3} states that we can always find a solution to \cref{eq:equiv_slope} with $k'$ and $l'$ both smaller in magnitude than $L$. These results are combined to obtain the uniqueness result in \cref{thm:2D_sq}.

\begin{theorem}
\label{thm:2D_sq}
\textcolor{red}{\rm Consider a generic binary matrix ${\tt X}$ of known dimension $N \times N$, where $N$ is prime. Let $L_0 = \lfloor \sqrt{N} \rfloor$. Then the inverse problem {\rm IP$(N,N,L,L)$} (see~{\rm \cref{def:IP}}) is uniquely solvable for any $L \ge L_0$.}
\end{theorem}

\begin{proof}
It is sufficient to prove the theorem for $L=L_0$. By the inverse DFT in \cref{DFT_inv_2D}, knowledge of all DFT coefficients uniquely determines any binary matrix. Suppose that $\tilde{X}_{kl}$ is unknown for some $k$ and $l$ such that $|k|$ or $|l|$ is greater than $L_0$. By \cref{lemma3}, there exists a $k'$ and $l'$ satisfying $|k'|,|l'|\le L_0$ and $(k',l')\sim(k,l)$. By \cref{eq:equiv_slope}, $\tilde{X}_{kl}$ and  $\tilde{X}_{k'l'}$ are dependent and provide identical information. As $\tilde{X}_{k'l'}$ is within the assumed pass band, ${\tt X}$ is uniquely determined.
\end{proof}

Approaching this setup geometrically, one can represent the entries of the $N \times N$ matrix as a $N \times N$ grid of points, and consider all the lines that (periodically) connect these points. This is an example of a finite affine plane of order $N$~\cite{hartshorne_book_2013}. It is known that each line in such a geometry contains $N$ points, and each point is on $N+1$ lines (with $N$ parallel classes for each line for a total of $N^2+N$ lines). As each DFT coefficient provides the popcount along $N$ lines in a parallel class, there can, in fact, only be $N+1$ independent DFT coefficients (in addition to the global popcount $\tilde{X}_{00}$). 

This observation implies that the condition provided by \cref{thm:2D_sq} is not a necessary one; it is sufficient but necessary to know {\em all} DFT coefficients up to order $L_0$ for unique recovery. However, the theorem states that at least one of the required $N+1$ coefficients (in addition to $\tilde{X}_{00}$) is of the order $L_0$. For example, for $N=17$, we have $L_0 = \lfloor\sqrt{17}\rfloor = 4$, but out of the $N+1 = 18$ coefficients needed (in addition to the global popcount) to guarantee recovery, $\tilde{X}_{1,4}$ and $\tilde{X}_{4,1}$ are the only independent coefficients of 4th order. All other 4th order DFT coefficients are equivalent to some coefficient of lower order by \cref{eq:equiv_slope}. One can see that, in general, uniqueness requires knowledge of at least one coefficient of the order $L_0$. This is so because $\tilde{X}_{L_0 1}$ is independent from all DFT coefficients of lower order. Indeed, there are no solutions to the equation $L_0 l' = k' \pmod N$ with $|k'|,|l'| < L_0$. 

\subsection{Square matrices of non-prime dimension}
\label{sec:uni.sqn}

When the dimension of a square binary matrix is not prime, the geometric interpretation of the coefficients is not as apparent. Consider a binary matrix of the size $N \times N$ where $N=p^\alpha$ with $\alpha > 1$. The DFT coefficients can be expressed in this case as
\begin{align}
\label{eq:case1}
\tilde{X}_{kl} = \sum_{m,n=1}^N X_{mn} \, e^{2\pi \cu \left( mk + nl \right) / N}
= \sum_{j=1}^N \ \eta_j \ e^{2\pi \cu \left( j / N \right) }\ , \ \ \mbox{where} \ \eta_j = \hspace{-8mm}
\sum_{\substack{m,n=1\\mk + nl = j \ ({\rm mod}\ N)}}^N \hspace{-8mm} X_{mn} \ .
\end{align}
The last expression partitions the entries of ${\tt X}$ according to $mk+nl = j \pmod{N}$ for each integer $j$ in the range $1\le j \le N$. We no longer refer to the entries satisfying  $mk + nl= j \pmod{N}$ as a line because this fails the usual geometric definition of two lines intersecting at most once. For example, for $\alpha=2$, the partition $k=0$, $l=j=1$ and the partition $k=p$, $l=j=1$ intersect at $(m,n)=(\mu p, 1)$ for all $0\le \mu \le p-1$. Moreover, as $N$ is not prime, the DFT coefficients no longer uniquely determine the sums along these partitions. In particular, the DFT coefficient $\tilde{X}_{10}$ no longer uniquely determines the row sums of ${\tt X}$. By \cref{lem:lenstra1}, it is possible that $\tilde{X}_{10} = 0$ as long as, for all $m$, the row sums $r_m$ satisfy $r_m = r_{m + \mu p^{\alpha-1}}$ for $0 \le \mu \le p-1$. It is straightforward to see that $\tilde{X}_{k0}$ yields identical information, as long as $k$ is not a multiple of $p$. When $k$ is a multiple of $p$, let $\beta = \log_p(\gcd(k, N))$. Then, for $k'=k/\gcd(k, N)$, we have
\begin{align}
\label{eq:tildeX_k0}
\tilde{X}_{k0} = \sum_{m,n=1}^{N} X_{mn} e^{2\pi \cu \left( m k / N \right) } = \sum_{m,n=1}^{N} X_{mn} e^{2\pi \cu \left( m k' / p^{\alpha-\beta} \right) } \ .
\end{align}
The second sum involves roots of unity of the order $p^{\alpha-\beta}$, each root appearing $p^{\alpha+\beta}$ times. Intuitively, this suggests that $\tilde{X}_{k_1 0}$ and $\tilde{X}_{k_2 0}$ contain different information if $\gcd(k_1, N) \neq \gcd(k_2, N)$.  For $k < N$, we have the bound $\gcd(k, N) \le p^{\alpha-1}$. This suggests that $X_{p^{\alpha-1},0}$ contains new information as compared to all the lower-order coefficients and motivates the uniqueness result in \cref{thm:2D_sq_alpha}.

\begin{theorem}
\label{thm:2D_sq_alpha}
\textcolor{red}{\rm Consider a generic binary matrix ${\tt X}$ of known dimension $N \times N$ where $N = p^\alpha$, $p$ is prime and $\alpha>1$ an integer. Define $L_0 = p^{\alpha-1}$. Then the inverse problem {\rm IP$(N,N,L,L)$} (see {\rm \cref{def:IP}}), is uniquely solvable for any $L\ge L_0$.}
\end{theorem}

Before proceeding, we state and prove the following useful lemma: 

\begin{lemma}
\label{lem:4}
\textcolor{blue}{\rm Under the conditions of \cref{thm:2D_sq_alpha}, let $k'l = kl' \pmod {N}$, and suppose that at least one of $k$ and $l$ is relatively prime with $p$. Then $X_{k'l'}=0$ implies that $X_{kl}=0$.}
\end{lemma}

\begin{proof}
Suppose that $X_{k'l'}=0$. Without loss of generality, assume that $\gcd(l,p)=1$. To employ \cref{lem:lenstra1} we first collect all powers of the $N$-th primitive root of unity $\zeta_N$. We rewrite this coefficient as 
\begin{align}
\label{eq:case1_ins}
0 = \tilde{X}_{k'l'} = \sum_{m,n=1}^N X_{mn} \, (\zeta_N)^{m k' + nl'}
  = \sum_{\mu=1}^N \left[
\sum_{\substack{m,n=1\\ mk' + nl' = \mu \ ( {\rm mod} \ N ) }}^N \hspace{-0.8cm} X_{mn} \right] \, (\zeta_N)^{\mu} \ .
\end{align}
By \cref{lem:lenstra1}, this implies that, for $1 \le \mu \le p^{\alpha-1}$ and for all $0 \le \nu \le p-1$,
\begin{align}
\label{eq:c_alpha}
\sum_{\substack{m,n =1 \\ mk' + nl' = \mu + \nu p^{\alpha-1} \ ( {\rm mod} \ N)}}^N \hspace{-1.5cm} X_{mn} = c_\mu
\end{align}
with some integer constant $c_\mu$. We need to prove that an identical expression holds for $\tilde{X}_{kl}$ for all $\mu$ and $\nu$ and a different set of constants,
\begin{align}
\label{eq:d_alpha}
\sum_{\substack{m,n=1\\mk + nl = \mu + \nu p^{\alpha-1} \ ( {\rm mod}\ N)}}^N \hspace{-1.5cm} X_{mn} = d_\mu \ .
\end{align}
For fixed $\mu$ and $\nu$, consider the indexes of terms summed in \cref{eq:d_alpha}. Using the fact that $k'l=kl' \pmod N$ and that $l$ has a multiplicative inverse, we make the following algebraic manipulations:
\begin{alignat*}{4}
&m k         &+ n l  &=         &\mu \ + \ &\nu            &p^{\alpha-1} &\pmod N  \\
&m kl^{-1}l' &+ n l' &= l^{-1}l'&\mu \ + \ &\nu ( l^{-1}l')&p^{\alpha-1} &\pmod N  \\
&m k'        &+ n l' &= l^{-1}l'&\mu \ + \ &\nu ( l^{-1}l')&p^{\alpha-1} &\pmod N  \ .
\end{alignat*}
Thus, letting $\mu'= l^{-1}l'\mu \pmod N$  and $\nu' = \nu (l^{-1}l') \pmod N$, we have 
\begin{align*}
\sum_{\substack{m,n=1\\mk + nl  = \mu  + \nu p^{\alpha-1} \ ( {\rm mod}\ N)}}^N \hspace{-12mm} X_{mn}  \hspace{10mm} = \sum_{\substack{m,n=1\\mk'+ nl' = \mu' + \nu'p^{\alpha-1} \ ( {\rm mod}\ N)}}^N \hspace{-12mm} X_{mn} \hspace{5mm} = c_{\mu'} \ ,
\end{align*}
where this last equality holds from \cref{eq:c_alpha}. Thus $d_\mu  = c_{\mu'}$ in \cref{eq:d_alpha}, which implies that $\tilde{X}_{kl}=0$. 
\end{proof}

\cref{lem:4} implies that $\tilde{X}_{k'l'}$ and $\tilde{X}_{kl}$ are dependent if $k'l=kl' \pmod {N}$. What remains to show is that this condition is satisfied for all DFT coefficients of order larger than $p^{\alpha-1}$. We are now ready to prove \cref{thm:2D_sq_alpha}.

\begin{proof}
	
\Cref{thm:2D_sq_alpha} will be proved by showing that, for any DFT coefficient $\tilde{X}_{kl}$ with either $|k|$ or $|l|$ greater than $L_0 = p^{\alpha-1}$, there exists a DFT coefficient $\tilde{X}_{k'l'}$ with $|k'|,|l'| \le L_0$ that already contains dependent information. We consider three separate cases: both $k$ and $l$ relatively prime with $N$, only one of $k$ and $l$ relatively prime with $N$, and neither $k$ nor $l$ relatively prime with $N$.
	
\paragraph{1) Case $\gcd(k,p) = \gcd(l,p) = 1$}
By a small extension of \cref{lemma3}, we can find $k'$ and $l'$ that are relatively prime with $p$, $|k'|,|l'| \le \lfloor\sqrt{p^\alpha}\rfloor< L_0$, and $(k,l) \sim (k',l')$, with the equivalence relation denoted by $\sim$ defined in (19). We can now apply \cref{lem:4} to these pairs of integers to conclude that $\tilde{X}_{k'l'}$ and $\tilde{X}_{kl}$ are dependent.
	
\paragraph{2) Case $\gcd(k,p)\neq 1, \gcd(l,p)=1$}
Without loss of generality, we will assume that $l$ is still relatively prime with $p$. Let $k'=\gcd(k,N)\le L_0$. With this choice of $k'$, we can find an $l'$ such that $|l'|\le L_0$ and $(k,l)\sim(k',l')$. As $k'l$ can take one of $N/k'$ values in $\mathbb{Z}_N$, and $k$ is an additive generator of these $N/k'$ values, there is some $l' \le N/k' \le L_0$ such that $kl'=k'l \pmod N$.  As $k'$ was chosen to be the greatest common divisor of $k$ and $N$, this choice of $l'$ must be relatively prime with $N$. Thus, \cref{lem:4} applies, implying that $\tilde{X}_{k'l'}$ and $\tilde{X}_{kl}$ are dependent.
	
\paragraph{3) Case $\gcd(k,p) \neq 1, \gcd(l,p) \neq 1$}
In this case, let $k'=\gcd(k,N)$ and $l' = \gcd(l,N)$, and without loss of generality, let $l' \le k'$. This case can be reduced to Case 1. Setting $\tilde{X}_{k'l'} = 0$, we have a vanishing sum of roots of unity of order $N/l'$
\begin{align*}
0= \tilde{X}_{k'l'} = \sum_{m,n=1}^N X_{mn} (\zeta_N)^{k'm+nl'} 
&= \sum_{m,n=1}^N X_{mn} (\zeta_{N/l'})^{(k'/l')m+n} \\
&= \sum_{\mu=1}^{N/l'} \left[\sum_{\substack{m,n=1\\(k'/l')m+n = \mu \pmod{N/l'}}}^N \hspace{-12mm} X_{mn} \hspace{4mm} \right] (\zeta_{N/l'})^{\mu} \ .
\end{align*}
This last equation is exactly the same as \cref{eq:case1_ins} in \cref{lem:4} with the substitutions $N\gets N/l'$, $k'\gets k'/l'$, and $l'\gets 1$. The result of \cref{lem:4} can now be applied, completing the proof.  
\end{proof}

The result of \cref{thm:2D_sq_alpha} is tight in the sense that there exist matrices that cannot be uniquely recovered with the data bandwidth $L < L_0=p^{\alpha-1}$. Unfortunately, this implies that we have no universal super-resolution (as defined in this paper) for square matrices of the size $N=2^\alpha$. By the even version of $\cref{eq:oddM}$, all DFT coefficients are in the range $[-2^{\alpha-1}+1, 2^{\alpha-1}]$. With $L_0 = 2^{\alpha-1}$, this range is equivalent to $[-L_0+1,L_0]$. Thus, the condition $L=L_0$ is equivalent to the requirement that the complete set of DFT coefficients be known.  As an example, consider the checkerboard matrices defined entry-wise by
\begin{align*}
X_{nm} = \begin{cases} 
0 & n+m = 1 \pmod 2 \\ 
1 & n+m = 0 \pmod 2
\end{cases} \ \ , \ \  
Y_{nm} = \begin{cases} 
0 & n+m = 0 \pmod 2 \\ 
1 & n+m = 1 \pmod 2
\end{cases} \ ,
\end{align*}
where $1 \le n,m \le 2^\alpha$. The corresponding DFT coefficients are given by 
\begin{align}
\label{eq:checker_DFT}
\tilde{X}_{kl} = 
\begin{cases} 
2^{\alpha-1} & k=l=0 \\ 
2^{\alpha-1} & k=l=2^{\alpha-1} \\ 
0 & \mbox{otherwise}
\end{cases} \ \ , \ \ 
\tilde{Y}_{kl} = 
\begin{cases} 
2^{\alpha-1} & k=l=0 \\ 
-2^{\alpha-1} & k=l=2^{\alpha-1} \\ 
0            & \mbox{otherwise}
\end{cases}\ .
\end{align}
These coefficient values can be readily obtained by letting ${\tt X} = \frac{1}{2}({\tt J} + {\tt A})$ where ${\tt J}$ is the matrix of all ones and ${\tt A}$ is the matrix with the entries $A_{nm} = (-1)^{n+m}$. The only nonzero DFT coefficient of ${\tt J}$ is $\tilde{J}_{00} = 2^\alpha$. Similarly, we can represent the entries of ${\tt A}$ as $A_{nm} = e^{\pi \cu (n + m)} = e^{2\pi \cu (m + n)(N/2)/N}$, which shifts the nonzero entry to the position $(N/2, N/2) = (2^{\alpha-1}, 2^{\alpha-1})$. Similar logic applied to ${\tt Y} = \frac{1}{2}({\tt J} - {\tt A})$ yields the expression given in \cref{eq:checker_DFT}. These two matrices agree on all coefficients except one that requires $L=L+0=2^{\alpha-1}$. 

Similarly to the previous checkerboard example, we can show that the square matrices of the size $N=p^\alpha$ (with $p$ being a prime greater than $2$) defined as
\begin{align*}
X_{nm} = 
\begin{cases} 1 & n+m = 1 \pmod p \\ 
0               & \mbox{otherwise} 
\end{cases} \ \ , \ \ 
Y_{nm} = 
\begin{cases} 
1 & n+m = 0 \pmod p \\ 
0 & \mbox{otherwise}
\end{cases}
\end{align*}
are $(p^{\alpha-1})$-indistinguishable, implying that the band width $L=p^{\alpha-1}$ is required for unique inversion. 

\section{Inversion algorithms}
\label{sec:inv}

We now discuss the algorithms to recover binary matrices for each case considered: (i) rectangular matrices with dimensions $N_1 \times N_2$ where $N_1$ and $N_2$ are distinct primes, (ii) square matrices of dimension $N \times N$ where $N$ is prime, and square matrices with $N$ of the form $N=p^\alpha$, where $p$ is prime and $\alpha>1$ an integer. For each case, we assume access to a large enough bandwidth of DFT coefficients to guarantee uniqueness, as determined by the previous section. 

\subsection{General strategy}
\label{sec:inv.gen}

Let, as above, the total number of elements in an $N_1 \times N_2$ matrix be denoted as $T = N_1 N_2$. Even under the conditions when each matrix of given dimension is, theoretically, uniquely determined by the data, finding the inverse solution by exhaustive search requires testing ${T \choose S}$ possibilities, where $S$ is the global popcount. Under the condition when $S\sim T / 2$, this strategy quickly becomes computationally prohibitive. However, inspired by the theoretical derivations shown above, we can break the inverse problem into more manageable steps and significantly increase the computational efficiency. Before developing algorithms for each case considered, we make an observation on the general form of these subproblems.

Theory suggests that the DFT coefficients often contain information equivalent to how many ones are present in each periodic line. For example, when $N_2$ is prime, $\tilde{X}_{01}$ (in conjunction with $S = \tilde{X}_{00}$) is equivalent to knowing how many ones are present in each column of ${\tt X}$. We thus consider the related combinatorial problem of placing $S$ ones in $N_2$ boxes, where we can place no more than $N_1$ ones in each box. By the inclusion-exclusion principle, one can compute the total number of possibilities as 
\begin{align}
\label{2D_comb1}
\sum_{n\ge 0} (-1)^n {N_1 \choose n}{S-nN_2+N_1-(n+1) \choose N_1-1} \ .
\end{align}
This formula gives the complexity of finding by exhaustive search the column-wise sums of ${\tt X}$. If this problem can be solved, the search space for the unique binary image has been significantly reduced to only those matrices with the correct number of ones in each column. We need to find among those the matrix that matches any remaining known but yet unused DFT coefficients. Refer to the correct column sum values as $c_n$ for $1 \le n \le N_2$. The unique binary image that matches the four given DFT coefficients is now within a set of size
\begin{align}
\label{2D_comb2}
\prod_{n=1}^{N_2} {N_1 \choose c_n} \ .
\end{align}

As a concrete example, consider the case $N_1=7$, $N_2=11$ and $S=38$. The number of distinct binary matrices with these parameters is $1.36 \times10^{22}$. The problem of determining the $c_n$ values is substantially smaller and is of size $1,528,688$ according to \cref{2D_comb1}. With only the $c_n$ values known, the overall search space has been reduced to an upper bound of $35^{11} \approx 9.65 \times 10^{16}$ by \cref{2D_comb2}. As $N_1$ is also prime here, one could repeat this process to further reduce the search space size by similarly solving for the row-wise popcounts $r_m$ -- which has a smaller individual problem size of $443,658,688$. The ensuing algorithms make use of these ideas to break down larger problems into more manageable subproblems. However, we still need methods that are more efficient than exhaustive search to solve these subproblems.

\subsection{Integer linear programming (ILP) and lattices}
\label{sec:inv.ilp}

Finding an $N_1 \times N_2$ binary matrix ${\tt X}$ that agrees with all available DFT coefficients can be phrased as an integer linear programming (ILP) problem of the form 
\begin{align}
\label{eq:intlin}
{\tt A}{\tt x} = {\tt b} \ , \ \ \   x_i \in\{0,1\} \ . 
\end{align}
In this formulation, ${\tt x}$ is a binary vector of length $N_1N_2$, which corresponds to stacking the columns of ${\tt X}$. The matrix ${\tt A}$ contains the relevant Fourier matrix entries, with ${\tt b}$ containing the available DFT coefficients. In line with \cref{DFT_dir}, we can express these entries using multi-indices of the form 
\begin{align*}
A_{(k,l),(m,n)} = e^{2\pi \cu \left( mk / N_1 + nl / N_2 \right)} \ , \ \  b_{(k,l)} = \tilde{X}_{kl} \ ,
\end{align*}
where the multi-index $(m,n)$ varies over $1 \le m \le N_1$ and $1\le n \le N_2$, and $(k,l)$ varies over the indexes corresponding to the available DFT coefficients. Note that, in an actual implementation, the entries of ${\tt A}$ and ${\tt b}$ are split into real and imaginary parts, which forces the entries of ${\tt x}$ to be real. Thus, if $M$ DFT coefficients are known in addition to $\tilde{X}_{00}$, then ${\tt A}$ is a $(2M+1) \times N_1N_2$ matrix, where we have taken into account that the row corresponding to $\tilde{X}_{00}$ has no imaginary part. For simplicity, we refer to ${\tt A}$ and ${\tt b}$ as having $M+1$ rows with complex entries. Additionally, no redundant coefficients (which are known to be conjugates of each other) are needed in an implementation.  For larger problems, ${\tt A}$ can be efficiently applied by fast Fourier transform techniques.  

Solving \cref{eq:intlin} is a known NP-hard problem. When using ILP techniques, as there is a unique solution, but no objective function to minimize, branch and bound methods do not offer significant improvement over exhaustive search. By defining an arbitrary objective function to minimize, the branch and bound may converge faster or slower, though it is typically difficult to tell {\em a priori} which is the case~\cite{aardal_2000_1}. Incorporating cutting planes and other preprocessing steps, however, can restrict the size of the search space~\cite{schrijver_1980_1, marchand_2002_1}. Without an objective function, ILP is reliant on these preprocessing steps to outperform exhaustive search.  As solving \eqref{eq:intlin} is NP-hard, the overall runtime is dominated by the size of the search space, as opposed to any cost of applying the matrix ${\tt A}$. 

An alternate approach to ILP is to use lattice basis reduction techniques. These techniques aim to reduce a given basis to short, nearly orthogonal vectors, with an end goal of facilitating calculations over the integers. We briefly summarize the celebrated  Lenstra-Lenstra-Lovasz (LLL) algorithm~\cite{lenstra_1982_1} for lattice basis reduction, which has many applications in mathematics and cryptography~\cite{hastad_1989_1}.  

Consider a linearly independent set of vectors $\mathbf{B}=\{{\tt b}_1,{\tt b}_2,\dots,{\tt b}_n\}$ in $\mathbb{R}^m$, where $n\le m$. The integer lattice $\mathcal{L}$ with this basis is the set of all linear combinations of the ${\tt b}_j$ with integer coefficients
\begin{align*}
\mathcal{L} = \{a_1{\tt b}_1+\cdots + a_n{\tt b}_n\ :\ a_j \in \mathbb{Z}\} \ .
\end{align*}
The LLL algorithm takes this basis of the lattice, $\mathbf{B}$, and returns a new basis $\mathbf{B^*}$, which is generally comprised of short, nearly orthogonal vectors. This basis $\mathbf{B}^*$ is called LLL-reduced, and is obtained through a Gram-Schmidt-like process, modified to ensure that the basis vectors stay in the lattice and to prioritize short vectors. Most importantly for our purposes, the first vector ${\tt b}^*_1$ in $\mathbf{B}^*$ will be the shortest in the new basis. It will not necessarily be the absolute shortest vector in the lattice~\cite{nguyen_2004_1}, but the LLL algorithm returns an approximately shortest vector in polynomial (hopefully, reasonable) time.

To see how we can use lattice reduction to solve \cref{eq:intlin} with $M$ known DFT coefficients, we first construct the $(T+M+1)\times(T+1)$ matrix (as before, $T=N_1 N_2$) with 4 blocks defined as
\begin{align}
\label{eq:LLL}
{\tt B} =   \left(\begin{array}{@{}c|c@{}}
{\tt I} & {\tt O} \\
\hline
\beta{\tt A} & - \beta {\tt b} 
\end{array}\right) \ .
\end{align}
In this $2\times2$ block matrix form, ${\tt A}$ and ${\tt b}$ are defined as in \cref{eq:intlin}, and ${\tt I}$ and ${\tt O}$ are the identity matrix and zero vector of the length $T$. The constant $\beta$ that appears in the lower two blocks is assumed to be large. Again, in an actual implementation, the ${\tt A}$ and ${\tt b}$ blocks would have $2M+1$ rows to account for real and imaginary parts.  

The LLL algorithm can now be performed on ${\tt B}$, treating the columns of the matrix as the lattice basis elements of length $T+M+1$. The shortest vector in the resulting LLL-reduced basis, ${\tt b}^*_1$, must necessarily be a linear combination of the original basis vectors. Letting ${\tt x} = [ a_1, a_2, \dots, a_T]$ be an integer vector, any vector in the lattice $\mathcal{L}$ is of the form
\begin{align*}
{\tt b}_1^* = [a_1, a_2, \dots, a_T | \beta({\tt A}{\tt x} - {\tt b})] \ .
\end{align*}
If $\beta$ is chosen to be sufficiently large, this shortest vector will likely minimize ${\tt A}{\tt x} - {\tt b}$, with the vector ${\tt x}$ being the proposed integer solution. Additional details of the algorithm can be found in~\cite{borwein_2000_1, ferguson_1999_1}. 

Finding the shortest vector in the lattice is also known to be an NP-hard problem.   The potential advantages of the LLL algorithm rely on the fact that it is an approximation algorithm,  and can be expected to find a solution in polynomial time~\cite{lenstra_1983_1}.   However, as an approximation algorithm,  there is no guarantee that it will outperform ILP techniques in general.   In fact, by changing parameters in LLL, one can trade off between a faster runtime and a higher probability of finding a sufficiently short vector. However, the runtime of the LLL algorithm is $\mathcal{O}(n^5m\log^3(B))$, where $B=\max_i \|{\tt b}_i\|_2$, which implies that, for practical purposes, the polynomial time still increases quickly in the size of the problem $n$ \cite{lenstra_1983_1}.  

One downside to the LLL algorithm is that it does not incorporate known bounds on the integer values. For example, if it is known that the correct integer values are either 0 or 1, the shortest vector in the LLL-reduced basis is not guaranteed to have binary coefficients. In contrast, ILP obeys the integer bounds throughout its search.  

Taking into account the relative advantages and disadvantages between these two approaches, we use a combination of ILP and LLL in the following algorithms. In general, the LLL algorithm was found to be much more efficient when running on problems with a smaller number of unknowns, which can take integer values in a possibly large range.  This takes advantage of the fact that LLL is independent of the known bound on the integers. In contrast, ILP depends heavily on the range of the integers, and can be more reliable when the integers are known to be binary. ILP can also be effective for large problems (with many constraints) when cutting planes can reduce the overall size. Anecdotally, ILP had slightly more stability than LLL when attempting to reconstruct with only $M=1$ DFT coefficient. 

\subsection{Algorithms}
\label{sec:inv.alg}

We now describe the algorithms for reconstructing the three cases of matrix dimensions. While the three algorithms share many similarities, we consider each case separately.

\subsubsection{Case when $N_1\neq N_2$ are both prime}

By the theoretical results for uniqueness, we assume access to only the 4 DFT coefficients $\tilde{X}_{00}, \tilde{X}_{01}, \tilde{X}_{10}$, and $ \tilde{X}_{11}$. As described in the beginning of this section, we first consider the smaller problem of using $\tilde{X}_{00}$ and $\tilde{X}_{01}$ to reconstruct the column sums of ${\tt X}$. Thus we consider the problem
\begin{align}
\label{eq:intlin_col}
{\tt A}^{(01)}{\tt c} = {\tt b}^{(01)} \ , \ \   0 \le c_i \le N_2 \ ,
\end{align}
where the unknown vector ${\tt c}$ represents the column sums of ${\tt X}$, and ${\tt A}^{(01)}$ and ${\tt b}^{(01)}$ refer to the respective sub-matrix of ${\tt A}$ and sub-vector of ${\tt b}$ containing only the rows corresponding to $\tilde{X}_{00}$ and $\tilde{X}_{01}$. The columns of ${\tt A}^{(01)}$ are similarly restricted to only have one representative entry from each column of ${\tt X}$. In line with the previous discussion, even though the number of unknowns has been greatly reduced from $N_1N_2$ to $N_1$, and the bound on the integers has been increased to $N_2$, solving this problem using ILP was preferable for stability reasons as $M=1$, where $M$ is the number of DFT coefficients corresponding to this directional sum. 

After finding the column sums ${\tt c}$ via \cref{eq:intlin_col}, we solve the corresponding problem
\begin{align}
\label{eq:intlin_row}
{\tt A}^{(10)}{\tt r} = {\tt b}^{(10)} \qquad ; \qquad  0\le r_i \le N_1 \ ,
\end{align}
to obtain the corresponding row sum vector ${\tt r}$. In \cref{eq:intlin_row}, the matrix ${\tt A}^{(10)}$ and ${\tt b}^{(10)}$ contain only the rows pertaining to $\tilde{X}_{00}$ and $\tilde{X}_{10}$. With this additional row information in hand, we finally solve the full binary system 
\begin{align}
\label{eq:intlin_col3}
\begin{bmatrix}
{\tt A} \\ - \\ {\tt C} \\ - \\{\tt R} \end{bmatrix}
{\tt x} = \begin{bmatrix} {\tt b} \\ - \\ {\tt c} \\ - \\ {\tt r} \end{bmatrix}\qquad ; \qquad  x_i\in\{0,1\} \ ,
\end{align}
where the $N_2\times N_1N_2$ binary matrix ${\tt C}$ contains ones appropriately to sum the column entries of ${\tt x}$. As ${\tt x}$ is formed by stacking the columns of ${\tt A}$, ${\tt C}$ is defined by
\begin{align}
\label{eq:col_matrix_def}
C_{mn} = \begin{cases} 
1 &  (m-1)N_1+1 \le n \le mN_1 \\
0 & \text{else}
\end{cases} \ .
\end{align}
The matrix ${\tt R}$ is defined similarly to ${\tt C}$ as in \cref{eq:col_matrix_def} to sum the rows of ${\tt X}$ based on the ordering of ${\tt x}$. For storage efficiency, one can remove the $(00)$, $(01)$, and $(10)$ rows from ${\tt A}$ and ${\tt b}$ in \cref{eq:intlin_col3}, as this information is already contained in the ${\tt C}$  and ${\tt R}$ matrix blocks. This remaining system finds the binary matrix, which matches the DFT coefficient $\tilde{X}_{11}$ in the reduced search space with given column and row sums. As this is a larger system with binary integer bounds, it is generally more efficient to solve by using ILP. This is summarized in \cref{algo:case1}.

\begin{algorithm}
\begin{algorithmic}[1]
\STATE {\bf Input}: DFT Coefficients $\tilde{X}_{00},\tilde{X}_{01},\tilde{X}_{10},\tilde{X}_{11}$
\STATE {\bf Output}: $N_1 \times N_2$ binary matrix ${\tt X}$
\vspace{3mm}
\STATE	Use ILP to reconstruct column sums ${\tt c}$ using \cref{eq:intlin_col}
\STATE	Use ILP to reconstruct row sums ${\tt r}$ using  \cref{eq:intlin_row}
\STATE	Use ILP to solve \cref{eq:intlin_col3} for binary matrix ${\tt X}$
\end{algorithmic}
\caption{\label{algo:case1} Reconstruction algorithm for $N_1 \times N_2$ binary matrices where $N_1\neq N_2$ are primes.}
\end{algorithm} 

We remark that, when $N_1\gg N_2$, it may be computationally faster to skip solving for the row sums as a separate subproblem. That is, immediately after solving \cref{eq:intlin_col}, one can solve an equation of the form \cref{eq:intlin_col3} without the ${\tt R}$ block. Similarly, if $N_2\gg N_1$, it may be prudent to ignore solving for the column sums as its own subproblem.   The overall runtime considerations of \cref{algo:case1} are governed by the size of the search spaces for each subproblem, as discussed in Section \ref{sec:inv.gen}.

\subsubsection{Case when $N_1=N_2=N$ where $N$ is prime}

We take a similar algorithmic approach for reconstructing square $N \times N$ binary matrices. We again reconstruct the row and column sums of the matrix via \cref{eq:intlin_col} and \cref{eq:intlin_row} but can utilize the additional available DFT coefficients (as required by \cref{thm:2D_sq}) to hopefully reconstruct  larger matrices in a stable manner. 

The matrix ${\tt A}^{(01)}$ in \cref{eq:intlin_col} was used to solve for the column sums, which were contained in the DFT coefficient $\tilde{X}_{01}$. For $L = L_0 =  \lfloor\sqrt{N}\rfloor$, the corresponding submatrix ${\tt A}^{(01)}$ contains additional rows corresponding to the available DFT coefficients $\tilde{X}_{01}, \ldots , \tilde{X}_{0L}$, which are all equivalent to the column sum information. These extra equations improve the reconstruction speed and stability of recovery. As we now have a moderately sized system with $M>1$  DFT coefficients that encode column sum information, this system is efficiently solved using the LLL algorithm. After reconstructing the row and column sums, instead of immediately attempting to match a binary matrix with given row and column sums to the remaining DFT coefficients, we repeat this process for additional directions. For example, an analogous ILP problem can be set up to solve ${\tt A}^{(11)} {\tt d}^- = b^{(11)}$ which solves for the sums along the diagonal lines of slope -1 (using the DFT coefficients $\tilde{X}_{11},\dots,\tilde{X}_{LL}$).  

This can be repeated for all $N+1$ directions. However, while the row, column, and diagonal directions (slopes of $\pm 1$) all have $L$ related coefficients, no other direction will have $L$ coefficients, with possibly many directions only having one related coefficient. This can have an adverse effect on the computational efficiency and stability of recovery. Thus, the LLL algorithm may fail to recover the directional sums for certain directions. As a check, if the resulting shortest vector is not sufficiently short (using a predefined error tolerance), we ignore that direction and only include its information as the DFT coefficient, as was done for $\tilde{X}_{11}$ in \cref{algo:case1}. In our implementation, we used the maximum norm ($\|{\tt e}\|_\infty = \max_i |e_i|$) to measure the magnitude of this shortest vector. For improved stability, we do not attempt to reconstruct the directional sums along directions with only $M=1$ DFT coefficient, and similarly include the DFT coefficient value as a constraint.

After attempting to solve for the directional sums along all $N+1$ directions (skipping any with $M=1$), we form an ILP problem of the form \cref{eq:intlin_col3}. A block is added for each successful directional recovery that sums the entries along those directions as in \cref{eq:col_matrix_def}. The corresponding rows from the ${\tt A}$ block can be removed, with the remaining rows of ${\tt A}$ corresponding to directions with unsuccessful recoveries. Pseudocode for this algorithm is provided in \cref{algo:case2}. 

Each call of the LLL algorithm roughly scales as $N^6$ (recall runtime is $\mathcal{O}(n^5m\log^3(B)$) in \cref{algo:case2}.  This rough estimate ignores the $B$ term, and sets $n=m=N$.  As \cref{algo:case2} calls the LLL algorithm up to $N+1$ times, the total runtime can be proportional to $N^7$.  In practice, for larger values of $M$, it is anticipated that this additional data will help the algorithm converge quicker.  The main idea of \cref{algo:case2} is that the final ILP step will run very quickly as the size of the search space will be drastically reduced. \

\begin{algorithm}
\begin{algorithmic}[1]
\STATE {{\bf Input}: DFT Coefficients $\tilde{X}_{kl}$ for all $|k|,|l|\le L_0=\lfloor\sqrt{N}\rfloor$}
\STATE {{\bf Input}: Error tolerance $\epsilon$}
\STATE {{\bf Output}: $N \times N$ binary matrix ${\tt X}$}
\vspace{3mm}
\FOR{$j=1$ \TO $N+1$} 
	\STATE{Collect $M$ available DFT coefficients corresponding to direction $j$}
	\IF{$M>1$}
			\STATE{Use LLL and the $M$ DFT coefficients to find the shortest vector corresponding to the reconstruction of directional sums ${\tt d}_j$}
			\IF{$\|${\rm shortest vector}$\|_\infty<\epsilon$}
				\STATE{Successful Recovery: Create corresponding block ${\tt D}_j$}
			\ENDIF
	\ENDIF	
\ENDFOR	
\STATE{Create matrix with blocks ${\tt D}_j$ for all values of $j$ corresponding to a successful recovery and ${\tt A}$ containing DFT terms for unsuccessful recoveries.}
\STATE{Create  right hand side vector with corresponding blocks ${\tt d}_j$ and the DFT coefficients.}
\STATE{Use ILP to solve for binary matrix ${\tt X}$ using this matrix and right hand side.}
\end{algorithmic}
\caption{\label{algo:case2} Reconstruction algorithm for $N \times N$ binary matrices where $N$ is prime.}
\end{algorithm}

\subsubsection{Case when $N_1=N_2=p^\alpha$ where $p$ is prime and $\alpha>1$}

For square matrices of the size $N=p^\alpha$, more care is required. We will focus on the case when $\alpha=2$, but similar ideas hold in theory for $\alpha>2$. For $N=p^2$, as seen in~\cref{eq:tildeX_k0}, $\tilde{X}_{0p}$ is equivalent to knowing how many entries in total are in the column numbers that are equal modulo $p$. Refer to these combined column sums as $C_j$ for $j=1,\dots,p$. The values of $C_j$ can be solved quickly using the LLL algorithm, as there are only $p$ unknowns as opposed to $p^2$ (where each $C_j$ is bounded above by $p^3$). After this information is recovered, the remaining coefficients of the form $\tilde{X}_{0k}$ for $k<p$ are equivalent to knowing the individual column sums. One can set up a linear system of the form \cref{eq:intlin_col} to solve for the column sums $c_i$, with an additional block containing the constraints already obtained from $\tilde{X}_{0p}$. These additional linear constraints are of the form
\begin{align*}
\sum_{i=j \pmod p}\hspace{-0.55cm}c_i = C_j\ , \ \ \mbox{for} \ j=1,\dots,p \ . 
\end{align*}
Identical results hold for the row sum by first using $\tilde{X}_{p0}$ and subsequently looking at \\$\tilde{X}_{10},\dots,\tilde{X}_{p-1,0}$. This is also true for the diagonal sums using $\tilde{X}_{pp}$ and $\tilde{X}_{-p,p}$, which are all in the available DFT coefficient range. However, one cannot simply recover the sums along other directions based on specific DFT coefficients. Consider any DFT coefficient $\tilde{X}_{kl}$ where at least one of $k$ and $l$ is relatively prime with $p$. By \cref{eq:case1}, this coefficient is still a sum of roots of unity of order $N=p^2$, where each root corresponds to matrix entries $X_{mn}$ that satisfy $mk+nl=\mu \pmod{N}$, for some integer $1\le\mu\le N$. However, in \cref{eq:case1}, since the roots of unity are no longer of prime order, the value of this sum does not uniquely determine the integer coefficients.  By \cref{lem:lenstra1}, the integer coefficients can differ by a fixed constant across entries that are equal modulo $p$, and still give the same sum.  

As an illustrative example, consider  a $3^2 \times 3^2$ binary matrix with $S=40$ nonzero entries. If we are given $\tilde{X}_{1,2} = e^{2\pi \cu /9}$, we can deduce that there is at least one nonzero entry in the partition of entries with $j = 1$ from \cref{eq:case1}, which gives the exact value $e^{2\pi \cu / 9}$. However, the remaining $39$ nonzero entries still need to be distributed among the $9$ partitions. With the knowledge that $\tilde{X}_{1,2} = e^{2\pi \cu / 9}$, this distribution is not unique, but must satisfy the condition that the sum of the $39$ corresponding roots of unity is $0$, in accordance with \cref{lem:lenstra1}.   Using the notation in \cref{eq:case1},  let $\eta_j$ be the number of ones contained in the $j$th partition.  The four linear constraints for this example $\tilde{X}_{1,2}$ are thus
\begin{align}
\label{eq:mu_constraints}
\eta_1 = \eta_4 + 1 = \eta_7 + 1 \ ; \ \ \eta_2 = \eta_5 = \eta_8 \ ; \ \  \eta_3 = \eta_6 = \eta_9 \ ; \ \ \sum_{j=1}^9 \eta_j = S \ . 
\end{align}
These constraints ensure that $\tilde{X}_{1,2}=e^{2\pi i/9}$ and that all 40 ones are placed in a partition.  
However, in these linear constraints, none of the $\eta_j$ are uniquely determined from just $\tilde{X}_{1,2}$. On the other hand, \cref{thm:2D_sq_alpha} indicates that these $\eta_j$ values will be uniquely determinable in the larger context of all available DFT coefficients. 

To solve for these linear constraints, in the general case we use the LLL algorithm without $\tilde{X}_{00}$, to find a short vector that fits the coefficient. From this possible solution, one can deduce the linear constraints similar to the the form of the first 3 equations of \cref{eq:mu_constraints}.

The proposed algorithm is thus similar to \cref{algo:case2}, but with a modification to take into account that we cannot uniquely determine the sum along lines in all directions. First, reconstruct the sums along the rows, columns, and diagonal directions. These sums are uniquely determinable, and should be reasonably stable since there are $p$ related DFT coefficients. Following this step, instead of finding other directional sums, we find linear constraints that the binary matrix satisfies along these directions. Finally, we search for a binary matrix that matches all these constraints and any remaining DFT coefficients. This algorithm is summarized in \cref{algo:case3}.  The runtime considerations of \cref{algo:case3} are similar to \cref{algo:case2}, where the LLL steps of the algorithm scale like $(p^2)^7$.

\begin{algorithm}
\begin{algorithmic}[1]
\STATE {\bf Input}: DFT Coefficients $\tilde{X}_{kl}$ for all $|k|,|l|\le p$
\STATE {\bf Input}: Error tolerance $\epsilon$
\STATE{\bf Output}: $p^2\times p^2$ binary matrix ${\tt X}$
\vspace{3mm}
\FOR{\textnormal{DFT Direction in}$\{\tilde{X}_{0k},\tilde{X}_{k0},\tilde{X}_{kk},\tilde{X}_{k,-k}\}$}
	\STATE Use LLL algorithm with the DFT coefficient $k=p$ to reconstruct sums $D$ along lines modulo $p$
	\STATE Use LLL algorithm with the DFT coefficients $1\le k<p$ and $D$  to reconstruct individual directional sums ${\tt d}_j$
\ENDFOR
\FOR{\textnormal{Remaining Direction}}
	\STATE Collect $M$ available DFT coefficients corresponding to direction $j$ 
	\IF{$M>1$}
		\STATE Use LLL with the $M$ DFT coefficients (without $\tilde{X}_{00}$ to find shortest vector 
		\IF{$\|${\rm shortest vector}$\|_\infty<\epsilon$}
			\STATE Successful Recovery: Create corresponding block ${\tt D}_j$ that contains linear constraints which shortest vector obeys
		\ENDIF
	\ENDIF
\ENDFOR
\STATE Create matrix with blocks ${\tt D}_j$ for all $j$ corresponding to a successful recovery and ${\tt A}$ containing DFT terms for unsuccessful recoveries
\STATE	Create  right hand side vector with corresponding blocks ${\tt d}_j$ and the DFT coefficients
\STATE	Use ILP to solve for binary matrix ${\tt X}$ using this matrix and right hand side
\end{algorithmic}
\caption{\label{algo:case3} Reconstruction algorithm for $p^2\times p^2$ binary matrices with prime $p$.}
\end{algorithm}

\subsection{Stability}
\label{sec:inv.stab}

The intermediate steps in the algorithms described in the previous section center on finding integer coefficients for a cyclotomic integer to equal a known value, within some precision. For example, in \cref{algo:case1}, one first attempts to reconstruct the column sums by finding a cyclotomic integer of prime order $p$ whose integer coefficients are bounded by prime $q$, that matches the value of $\tilde{X}_{01}$. In \cref{algo:case2}, the same problem is considered, although it can be for one of $p+1$ potential directions, with possibly more than one corresponding DFT coefficient. Therefore, the key question when it comes to stability is how close can two distinct cyclotomic integers be to one another? 

Consider two distinct cyclotomic integers $A = \sum_{j=1}^p a_{j} \left( \zeta_{p} \right)^{j}$ and $B = \sum_{j=1}^p b_{j} \left( \zeta_{p} \right)^{j}$. Define $e_j = a_{j} - b_{j}$ so that $E = A - B = \sum_{j=1}^p e_{j} \left( \zeta_{p}\right)^{j}$. We wish to estimate how close $E$ can be to the origin of the complex plane.  Finding the exact solution to this problem is difficult~\cite{myerson_1986_1, habegger_2018_1}. However, we can provide a heuristic estimate. This will yield some insight towards the level of stability we can expect when reconstructing directional sums.  

Consider a direction with $M$ available corresponding DFT coefficients. These DFT coefficients are of the form 
\begin{align*}
\left[ \sum_{j=1}^p a_{j} \left(\zeta_{p} \right)^{j}, \sum_{j=1}^p a_{j} \left( \zeta_{p} \right)^{2j}, \dots, \sum_{j=1}^p a_{j} \left( \zeta_{p}\right)^{M j } \right] \ .
\end{align*}
If the coefficients $a_j$ and $b_j$ are bounded between $0$ and $K$, the coefficients $e_j$ satisfy $-K \le e_j \le K$. Moreover, as the total popcount $S = \tilde{X}_{00}$ is known, we have $\sum_{j=1}^p e_j = 0$. If the process of finding integer coefficients that agree with all the available DFT coefficients is unstable, then it is possible that all entries of the vector 
\begin{align*}
{\tt e} = \left[ \sum_{j=1}^p e_{j} \left(\zeta_{p}\right)^{j}\  , \ \  \sum_{j=1}^p e_{j} \left(\zeta_{p}\right)^{2j} \ , \ \ \dots \ , \ \  \sum_{j=1}^p e_{j} \left( \zeta_{p} \right)^{M j } \right]
\end{align*}
are small. For small $R$, we will determine an approximate condition for which $\|{\tt e}\|_\infty \le R$. Let $\rho(R)$ be the expected number of valid vectors ${\tt e}$ satisfying $\|{\tt e}\|_\infty \le R$. We model each term of the form $\sum_{j=1}^p e_{j} ( \zeta_p )^{kj}$ in ${\tt e}$ as a sum of $\sum_{j=1}^p |e_{j}|$ uniform random points on the unit circle. As $n \to \infty$, the probability that a sum of $n$ random points on the unit circle has length at most $R$ approaches $1 - \exp\left( -R^2/n \right)$~\cite{greenwood_1955_1}. The $x$- and $y$-coordinates approach independent normal distributions with the standard deviation $\sqrt{n/2}$ by the central limit theorem. Using this approximation, the probability that all entries of ${\tt e}$ are less than or equal to $R$ is $\left[ 1 - \exp\left(-R^2/\sum_{j=1}^p|e_j| \right) \right]^M$. By linearity of expectations, we can approximate the expectation  that $\|{\tt e}\|_\infty<R$ as a sum over all valid choices of $e_{j}$, viz,
\begin{align}
\label{eq:E_R}
\rho(R) \approx & \sum_{\substack{|e_j|<K\\\sum_{j=1}^p e_j = 0}}\left(1-e^{-R^2/\sum_{j=1}^p|e_j|}\right)^{M} \ .
\end{align}
Since $R$ is small, we approximate each term inside the sum using the linearization $e^x\approx 1+x$.
Moreover, similar to \cref{2D_comb1}, by the inclusion-exclusion principle, one can compute the total number of terms in this sum to be
\begin{align*}
\nu(p,K)= \sum_{n\ge 0}(-1)^n{p\choose n}{Kp-n(2K+1)+p-1\choose p-1} \ .
\end{align*}
By setting $\sum |e_{j}|=(2K+1)p/4$, which is roughly its average value, in \cref{eq:E_R}, and replacing the sum over its $\nu(p,K)$ choices of $e_j$, we have the reduced approximation
\begin{align}
\label{eq:E_RKMp2}
\rho(R) \approx \nu(p,K) \left[ 4R^2 / (2K+1)p \right]^M \ . 
\end{align}
The only solution to $\sum_{j=1}^p e_j ( \zeta_p)^{kj} = 0$ that also satisfies $\sum_{j=1}^p e_j =0$ is $e_j=0$. Thus we expect that, for small enough $R$, $\rho(R) \approx 1$. Setting this equal to our approximation \cref{eq:E_RKMp2} and solving for $R$, we find  
\begin{align*}
R^{2M} \approx (2K+1)^{M}(p/4)^M[\nu(p,K)]^{-1} \ ,
\end{align*}
so that we expect to require roughly
\begin{align}
\label{eq:stability1}
- \log(R) = \log(2)-\log\left(\sqrt{p(2K+1)}\right) + \log(\nu(p,K))  / 2M \ .
\end{align}
digits of precision to distinguish integer coefficients for the $M$ cyclotomic integers.

We emphasize that the result \cref{eq:stability1} is an approximation, and may not be accurate for small $p$. A more careful analysis would additionally account for the fact that sums of few points on the unit circle are significantly more likely to be small. However, when $K^2$ is large compared to $p$ (which is typical in applications), this contribution becomes negligible. So, we content ourselves with the above heuristic, keeping in mind that it may underestimate the precision needed.

\section{Numerical examples}
\label{sec:sims}

We next conduct numerical simulations to test the proposed recovery algorithms. In our implementation of all three algorithms, we use MATLAB's built-in solver for ILP, \texttt{intlinprog}, which uses cutting planes and other preprocessing steps to reduce the size of the computational domain. The prescribed stopping condition for any call of \texttt{intlinprog} was set to checking $10^7$ possible matrices.  

\begin{wraptable}{R}{5.0cm}
\centering
\begin{tabular}{llll}
$N_1$ & $N_2$ & $t$, sec. & $n_{\tt d}$ \\
\hline
\hline
5     & 7     & 0.05  & 2 \\
5     & 11    & 8     & 4 \\ 
5     & 13    & 10    & 5 \\
\hline
7     & 11    & 64    & 5 \\
7     & 13    & 84    & 6 \\
7     & 17    & 111   & 8 \\
\hline
11    & 13    & 91    & 7 \\
\hline
\end{tabular}
\vspace{2mm}
\caption{\rm \label{tab:1} Time $t$ for reconstructing $N_1\times N_2$ binary matrices with $S=\lfloor N_1 N_2 /2\rfloor$ nonzero entries using \cref{algo:case1}. Averages for $30$ randomly-generated model matrices are displayed. The column $n_{\tt d}$ displays the number of digits needed for stable recovery of the column sums with $M=1$ according to \cref{eq:stability1}.}
\end{wraptable}

The implemented LLL algorithm was programmed in MATLAB. After sufficient testing, the large constant parameter in \cref{eq:LLL} was set to $\beta =10^8$. The error tolerance to determine if a vector is sufficiently short was $\epsilon = 0.001$. In our implementations, we made one modification for practical time considerations. Some runs of the LLL-algorithm can take a very long time and ultimately fail to recover a sufficiently short vector. To avoid waiting too long for a failed recovery, we set a time limit on the LLL algorithm to 5 seconds. This stopping criterion was found to be a good balance between minimizing the computation time while not overlooking any feasible reconstructions. All computations were carried out in double precision.

\subsection{\cref{algo:case1} for rectangular matrices}
\label{sec:sims.1}

For different prime values of $N_1$ and $N_2$, a model $N_1\times N_2$ binary matrix generated with $S=\lfloor N_1N_2/2\rfloor$ nonzero entries, chosen uniformly at random. \cref{algo:case1} was run given the 4 DFT coefficients $\tilde{X}_{00}$, $\tilde{X}_{10}$, $\tilde{X}_{01}$, and $\tilde{X}_{11}$, to try to recover the original binary matrix exactly. For fixed values of $N_1$ and $N_2$, this experiment was run $30$ times, with the average timing (in seconds) displayed in \cref{tab:1}.  

\begin{wrapfigure}{r}{9.0cm}
\centering
\includegraphics[width=4.54cm]{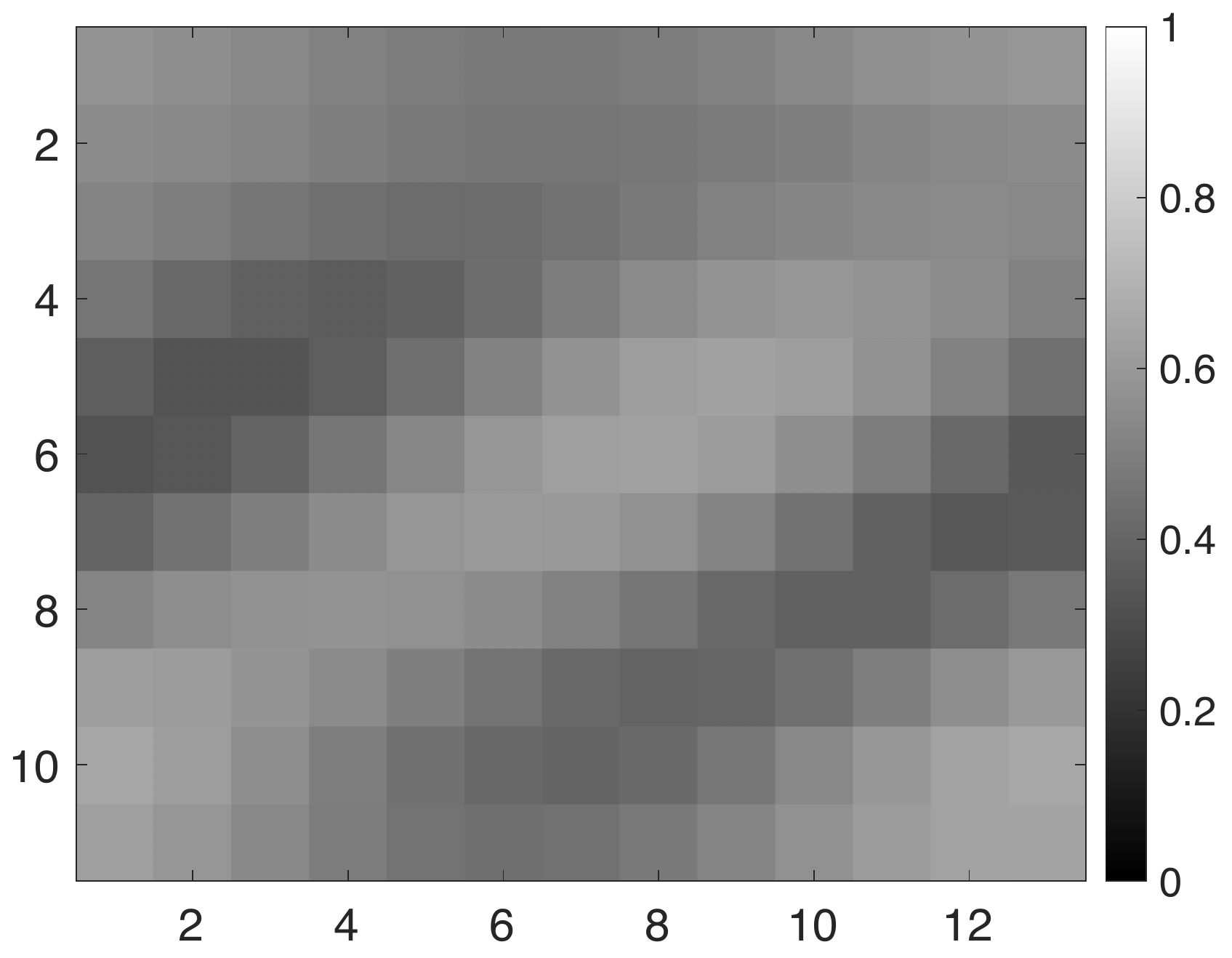} \  \includegraphics[width=4.1cm]{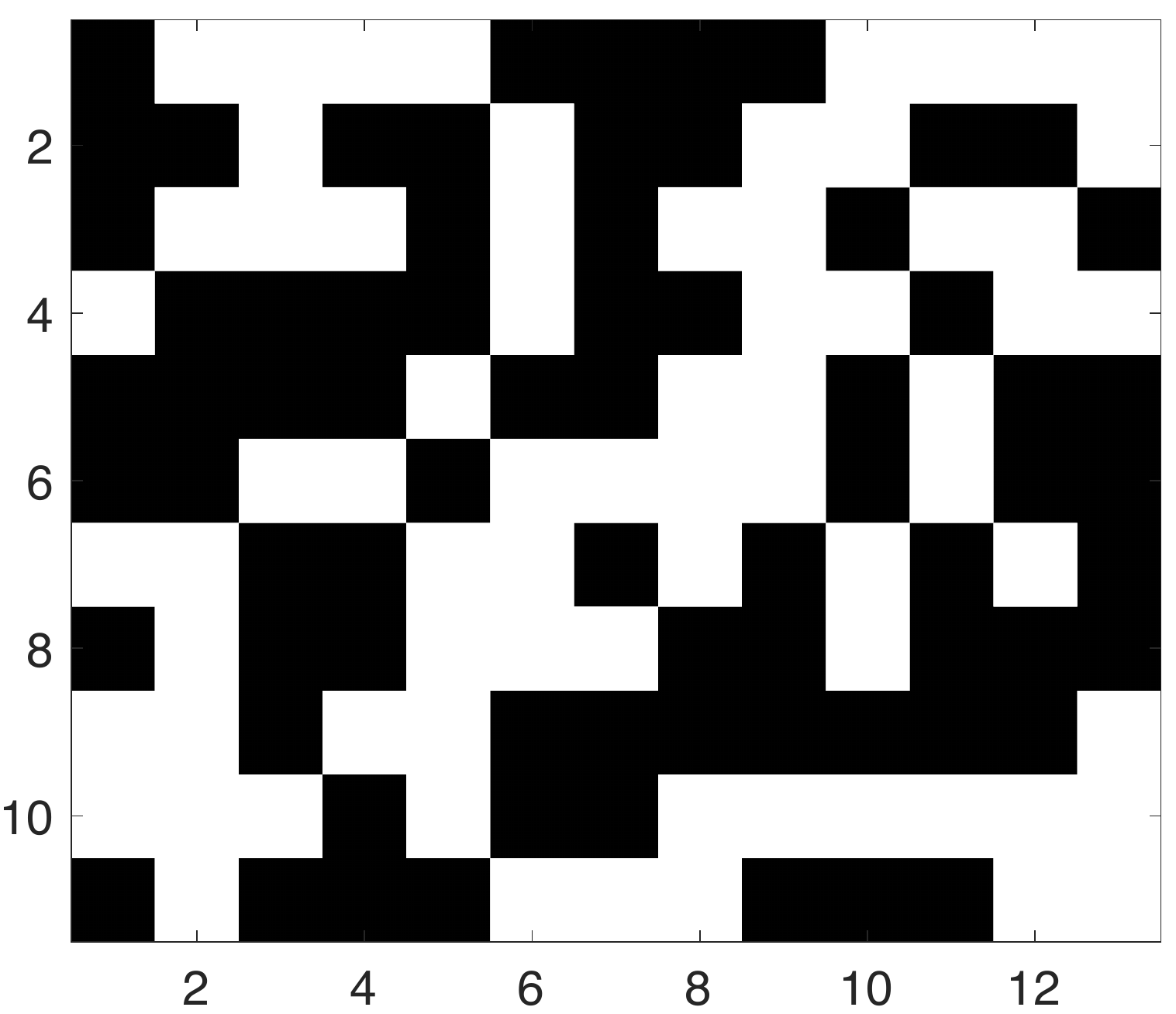}
\caption{\rm \label{fig:results1} Left: reconstruction with only $4$ available DFT coefficients (as required by \cref{thm:2D_rect})  of a randomly generated binary $11 \times 13$ matrix with $S=71$ ones. Right: reconstruction obtained by \cref{algo:case1} (coincides exactly with the model).}
\end{wrapfigure}

\Cref{algo:case1} was able to consistently recover the original binary matrix for dimensions as large as $11\times13$. A recovery is considered successful if it reconstructs all elements of the matrix correctly. For all of the dimensions displayed in \cref{tab:1}, \cref{algo:case1} was successful in all 30 trial runs.  Matrices of dimension $11 \times 13$ were the largest that could be reliably recovered within the prescribed stopping criteria, which took on average one and a half minutes. Note that $7\times17$ matrices have fewer elements but tended to take longer to be recovered due to the larger column dimension making the row sum recovery more computationally demanding.

As a comparison, we note that naively running ILP on the entire system \cref{eq:intlin}, as opposed to first considering the subproblems of recovering row and column sums, took on average 2.3 seconds (about 40 times longer than by \cref{algo:case1}) for $5\times 7$ matrices, and was unable to scale to $5 \times 11$ matrices under the prescribed stopping conditions.  

A band-limited reconstruction and the reconstruction by \cref{algo:case1} of a sample $11 \times 13$ binary matrix are shown in \cref{fig:results1}. The band-limited reconstruction ${\tt X}^{\tt (blurred)}$ is given by
\begin{align}
\label{eq:bl}
X_{mn}^{\tt (blurred)} = \frac{1}{N_1 N_2} \sum_{k=-L_1}^{L_1} \sum_{l=-L_2}^{L_2} \tilde{X}_{kl} e^{-2\pi\cu \left(mk / N_1 + nl / N_2 \right)} \ ,
\end{align}
which is identical to \cref{DFT_inv_2D} except the summation now runs only over the available pass band with parameters $L_1$ and $L_2$.   For this $N_1\times N_2$ case,  we set $L_1=L_2=1$ for the band-limited reconstruction in \cref{eq:bl} , but in accordance with \cref{thm:2D_rect}, the terms with $\tilde{X}_{1,-1}$ and $\tilde{X}_{-1,1}$ are removed from the sum (in the reconstruction, these DFT coefficients are not used).  It can be seen that, with such few data points, the band-limited reconstruction has little resemblance to the original matrix. However, a reconstruction that takes into account the matrix binarity returns the model exactly 

According to \cref{eq:stability1}, the subproblems of reconstructing the row- and column-wise sums is stable when working in double precision, even with only one DFT coefficient (either $\tilde{X}_{01}$ or $\tilde{X}_{10}$).  The number of digits after the decimal place estimated by the heuristic for $M=1$ for reconstructing the column sums (the more difficult direction) are given in the last column of \cref{tab:1}. For example, for $11 \times 13$ matrices, the heuristic suggests that we need about 7 digits to stably reconstruct the column sums. This corresponds to noise level of the magnitude of $\sim10^{-8}$ relative to $\tilde{X}_{10}$, which is of the order of unity. If we reduce the number of known digits to 6 (noise level of $\sim 10^{-7}$), the column sums for the 30 randomly generated models are reconstructed correctly in 6 cases. This strong instability can be rectified by including more DFT coefficients in the data set, beyond the minimum required for theoretical uniqueness. For the band limit parameter $L=2$, we have $M=2$ available DFT coefficients for reconstructing the column sums ($\tilde{X}_{10}$ and $\tilde{X}_{20}$). In this case, the column sums for all 30 matrices are recovered correctly with only three significant digits in the data.

\begin{wraptable}{r}{8cm}
\centerline{
\begin{tabular}{rrcrrc}
$N$    &$L$ & Rec., \% & $t$, sec. & Dir. & $n_{\tt d}$  \\
\hline\hline
17     & 4    & 100    & 2      & 14   &  5    \\
\hline
19     & 4    & 99     & 8      & 11   &  6    \\
19     & 5    & 100    & 3      & 20   &  $|$  \\
\hline
23     & 4    & 0      & --     & 8    &  8    \\
23     & 5    & 100    & 3      & 20   &  $|$  \\
\hline
29     & 5    & 96     & 5      & 16   &  11   \\
29     & 6    & 100    & 12     & 24   &  $|$  \\
\hline
\vspace{2mm}
\end{tabular}}
\caption{\rm \label{tab:2} Summary of reconstruction results for $N \times N$ binary matrices with global popcount $S=\lfloor N^2/2\rfloor$ using \cref{algo:case2}. The parameter $L$ indicates the pass band used. The next column displays the percentage of exact recoveries for 100 randomly-generated model matrices. The next two columns display the average timings and the average number of recovered directional sums. The last column displays the number of digits estimated by \cref{eq:stability1} that are needed for stable recovery of a directional sum with $M=2$ DFT coefficients. Reconstruction that took longer than the prescribed stopping condition is denoted by a dash.} 
\end{wraptable}

\subsection{\Cref{algo:case2} for $N \times N$ matrices with prime $N$}
\label{sec:sims2}

\Cref{algo:case2} was run on $100$ randomly-generated $N \times N$ binary matrices, for $N = 17, 19, 23, 29$. In each case, the global popcount was set to $S=\lfloor N^2/2\rfloor$, which is the most difficult case. The results of the simulations are summarized in \cref{tab:2}, which contains the average run time of the algorithm, the percentage of model matrices that were exactly recovered, and how many directional sums (out of $N+1$) were recovered on average by the LLL algorithm. Additionally the last column displays the stability estimate \cref{eq:stability1}, in terms of the number of digits in the data, for the most unstable directional sum recoveries with $M=2$ DFT coefficients (as any with $M=1$ are automatically skipped).

For $N=17$, when $L=\lfloor \sqrt{17}\rfloor=4$, the algorithm was able to reconstruct all 100 models in an average of 2 seconds. This is substantially faster than the implementation of \cref{algo:case1}, as we are now using a larger bandwidth of available DFT coefficients in accordance with the theory. The larger bandwidth provides more coefficients than are minimally required for uniqueness, which improves computational speed and stability. For example, we now have access to $\tilde{X}_{20}$, which provides equivalent information to $\tilde{X}_{10}$. This increased stability allows us to use LLL algorithm, which runs much faster than ILP. Out of the $N+1=18$ possible directions, 4 directional sums are skipped in the algorithm for having only 1 corresponding DFT coefficient.  On average 13.98 (this number is rounded off as 14 in \cref{tab:2}) of the remaining 14 directions were reconstructed accurately. Note that the final ILP step of the algorithm finds the unique solution quickly as is not a bottleneck.

As we increase the dimensions to $N=19$, but keep $L=4$, the average run time increases to about 8 seconds. There are now 8 directions that are skipped due to having only 1 DFT coefficient, and the algorithm reconstructs 11 of the remaining 12 directions on average. Most notably, we have our first instance of failed reconstruction where exactly 1 model matrix was not reconstructed accurately (out of 100). The algorithm in this case fails by ILP reporting that the linear system is inconsistent over the integers. Upon closer investigation, it is seen that the inconsistent system is caused by one of the LLL solves finding an incorrect directional sum due to an instability -- it found a sufficiently short vector, but not the correct one. Even though the stability heuristic suggests that 5 digits should be enough for stability, this is not the case for this model. It is not altogether surprising that there is an outlier, as the heuristic was based on statistical arguments. The fast notification of failure by the algorithm is important, as it did not return a misleading answer. This reconstruction could be remedied by removing one of the reconstructed directional sums by trial and error until ILP runs successfully. Another alternative for reconstructing this failed model is to improve stability by increasing the number of available DFT coefficients. When $L$ is increased to 5, which is more than required for uniqueness, all 100 models are reconstructed, in an average of under 3 seconds, where all $N+1=20$ directions are almost always reconstructed.  

The $N=23$ case is an interesting example. With $L=4$, only $8$ out of the $24$ directions have more than $1$ corresponding DFT coefficient. These $8$ directional sums are accurately reconstructed for each model. However, this does not provide enough information for making the final ILP step and finding the unique solution before the prescribed stopping criteria. Stability issues prevent reconstruction of the correct matrix if we remove this restriction on directional sums with only one coefficient (\cref{eq:stability1} suggests that about $17$ digits are required for $M=1$). If we increase $L$ and take $L=5$, the algorithm works for all models.  

\begin{wrapfigure}{r}{9cm}
\centering
\includegraphics[width=4.57cm]{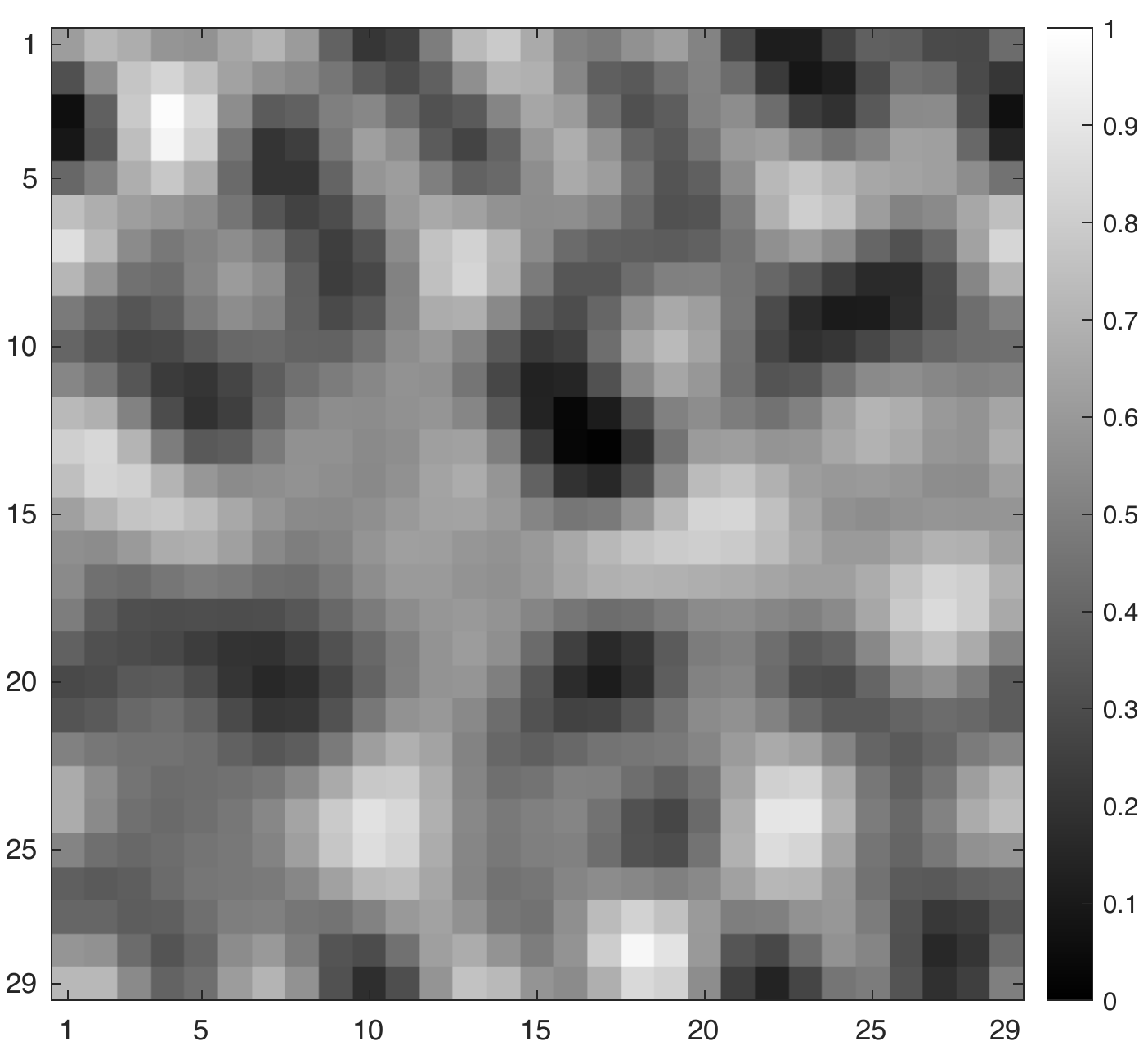}\  \includegraphics[width=4.16cm]{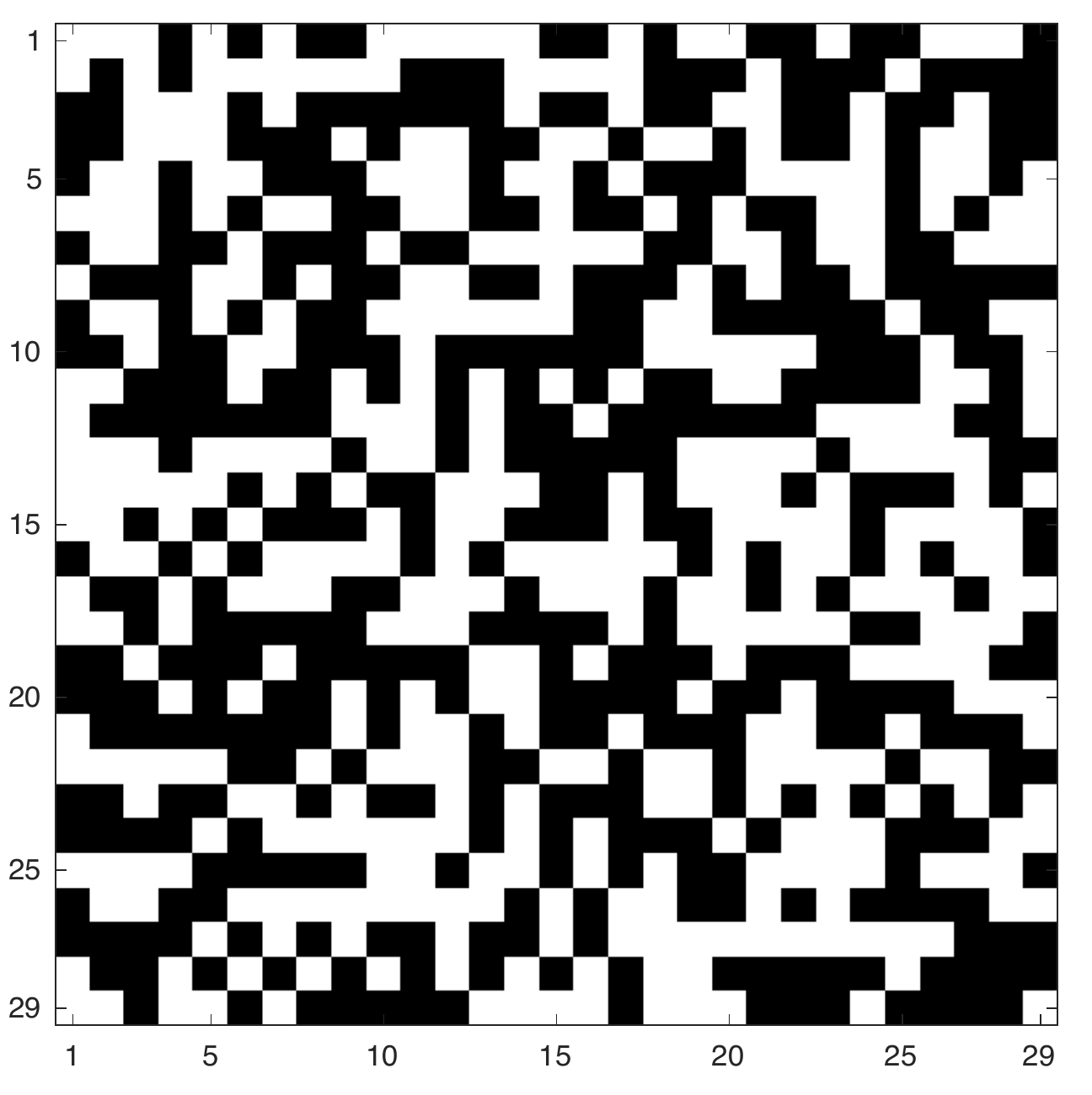}
\caption{\rm \label{fig:results2} Left: band-limited reconstruction with $L=5$ of a randomly-generated binary $29 \times 29$ matrix with the global popcount $S=420$. Right: reconstruction by \cref{algo:case2} (exactly coincides with the model).}
\end{wrapfigure}

The case $N=29$ has similar behavior to the $N=19$ case. At the minimal band limit parameter $L=5$, the algorithm almost always successfully recovers the model binary matrix, failing for 4 out of the 100 models. A band-limited reconstruction and the reconstruction by \cref{algo:case2} with $L=5$ of a sample model $29 \times 29$ binary matrix are shown in \cref{fig:results2}. When the algorithm fails, it fails, as above, by an unstable LLL step that causes an inconsistency in the ILP step. This can, again, be remedied by increasing $L$ to $6$. This increases the average time from about 5 seconds to 12 seconds, but recovers more directions on average ($24$ as opposed to $16$). 


As an example of reconstruction with noisy data, we have added Gaussian white noise to the DFT coefficients of the model in \cref{fig:results2} with variance $10^{-4}$ which only corrupted the DFT coefficients beyond 3 digits past the decimal point. Reconstruction failed until $L$ was increased to $L=9$.  At this bandwidth, 21 out of the 30 directional sums were recovered and the model was exactly reconstructed.  Importantly, the smallest number of DFT coefficients for any direction is now $M=4$. This value of $M$ requires 5 digits for stability according to \cref{eq:stability1}. However, the reconstruction in this case outperforms the heuristic.

Based on the success of recovering random $29 \times 29$ binary matrices, as a motivated example we seek to recover a blurred QR code. A $29 \times 29$ QR code that encodes the phrase ``{\tt DiscreteFourierTransform}" was generated according to the standard format specifications, known as a Version 3 QR code for this size. With the minimum bandwidth required for unique recovery $L=5$, \cref{algo:case2} was run on this incomplete set of DFT coefficients. Note that no additional QR code information was used -- the image was treated by the algorithm as a general binary matrix. For example, Version 3 QR codes have fixed patterns, including the recognizable position detector patterns present in three of the corners. Even though these fixed patterns are known based on the size of the QR code, the algorithm treats these as general regions which need to be reconstructed. This QR code information could certainly be added to the algorithm to improve computational speed and stability. The blurred QR code and its reconstruction using \cref{algo:case2} (which exactly recovers the original code) are displayed in \cref{fig:results2_1}. The reconstruction was done in about 6 seconds, with 15 out of the possible 30 directions recovered before the ILP solve.

\begin{wrapfigure}{r}{9cm}
\centering
\includegraphics[width=4.59cm]{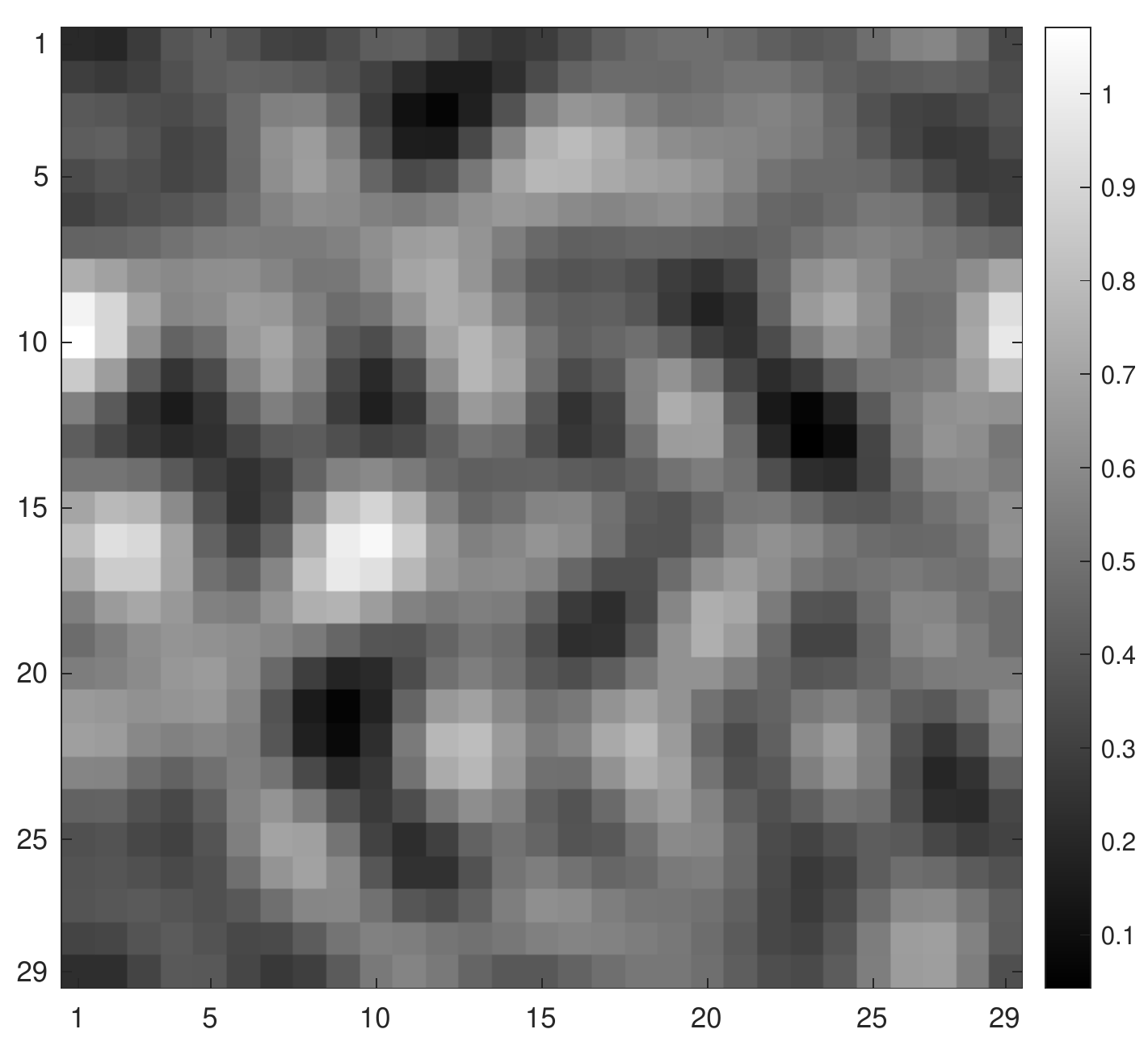}\  \includegraphics[width=4.13cm]{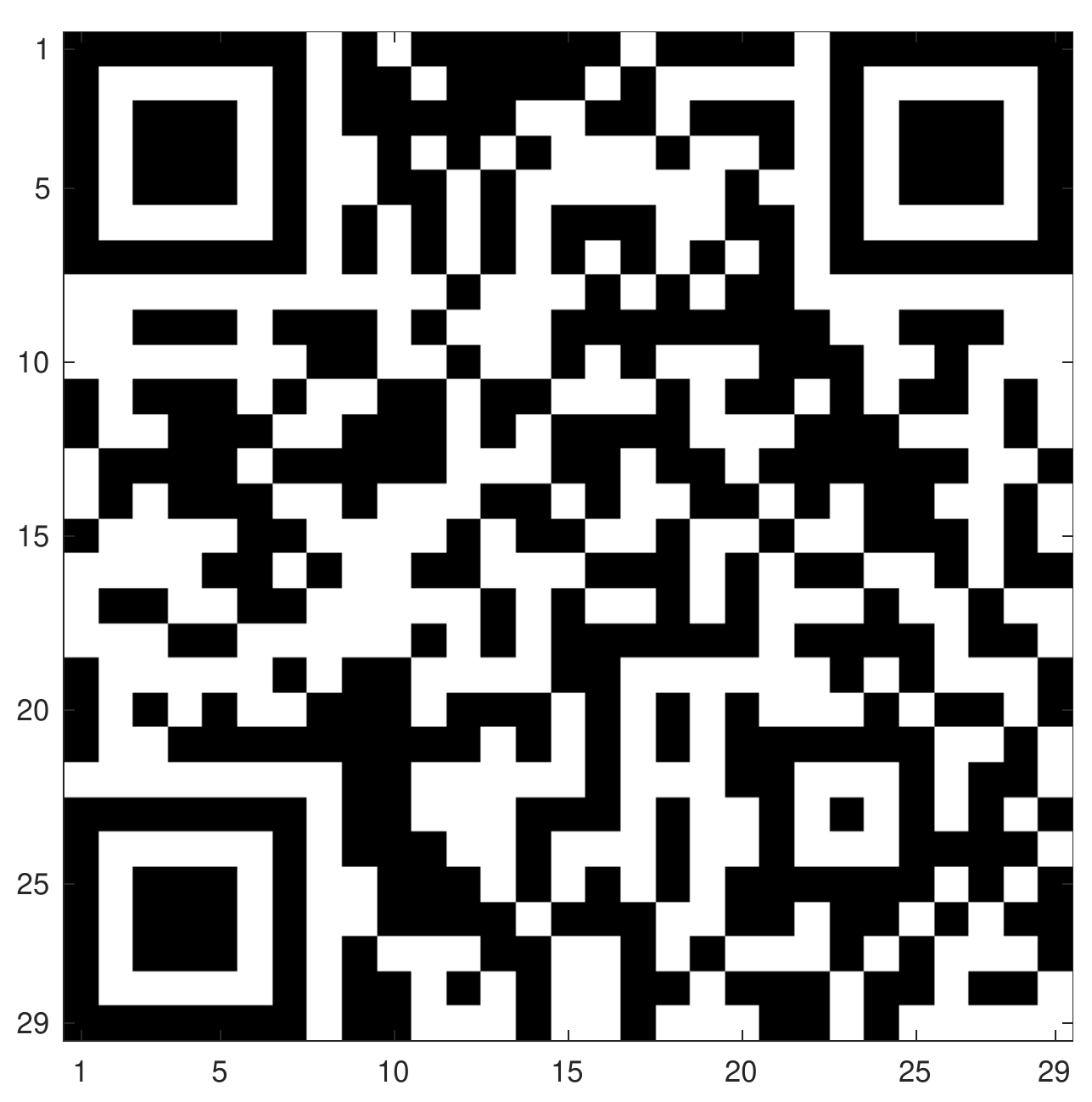}
\caption{\rm \label{fig:results2_1} Left: band-limited reconstruction with $L=5$ of a $29\times 29$ Version 3 QR code with $S=410$ nonzero entries. Right: reconstruction by \cref{algo:case2} (exactly coincides with the model).}
\end{wrapfigure}

\subsection{\Cref{algo:case3} for $p^2\times p^2$ matrices with prime $p$}

Finally, \cref{algo:case3} was tested on $25\times25$ binary matrices. With the available computational resources, \cref{algo:case3} was unable to scale to the next prime power of $49\times49$. Similar to the experiment performed for \cref{algo:case2}, we tested the algorithm on 100 randomly generated binary matrices in the most computationally difficult regime of $S=\lfloor 25^2/2 \rfloor = 312$ nonzero entries.  

With all DFT coefficients within the band limit defined by $L=5$, \cref{algo:case3} was able to exactly reconstruct the randomly generated binary matrix 87 out of 100 times in an average of about 25 seconds. This average timing includes both successful and failed recoveries. It is understandable that \cref{algo:case3} performed slightly worse than \cref{algo:case2}, as we can only reconstruct certain linear constraints for many of the directions for $p^2\times p^2$ matrices, as opposed to the directional sum values themselves. In all 100 simulations, the algorithm correctly recovered the only directional sums that are determinable: row, column, and diagonal directions. There were 26 remaining directions, with 12 of these directions automatically skipped for having only one corresponding DFT coefficient. Of the remaining 14 directions, the algorithm successfully found constraints (as measured by finding a corresponding sufficiently short vector) for 10 of these, on average.  Whenever the algorithm failed, it was again due to the ILP step finding an inconsistent system, which was caused by an instability (incorrect solve) in finding constraints for one of the directions.  

As a final practical test, the phrase ``{\tt Binary Matrix Recovery}'' was encoded in a $25\times 25$ Version 2 QR code. The true binary image has $S=287$ nonzero entries. With access to the DFT coefficients inside the bandwidth of $L=5$, \cref{algo:case3} was able to exactly reconstruct the original QR code in about 24 seconds. This reconstruction and the corresponding band-limited (blurred) image are displayed in \cref{fig:QR25}.

\begin{wrapfigure}{R}{9cm}
\centering
\includegraphics[width=4.59cm]{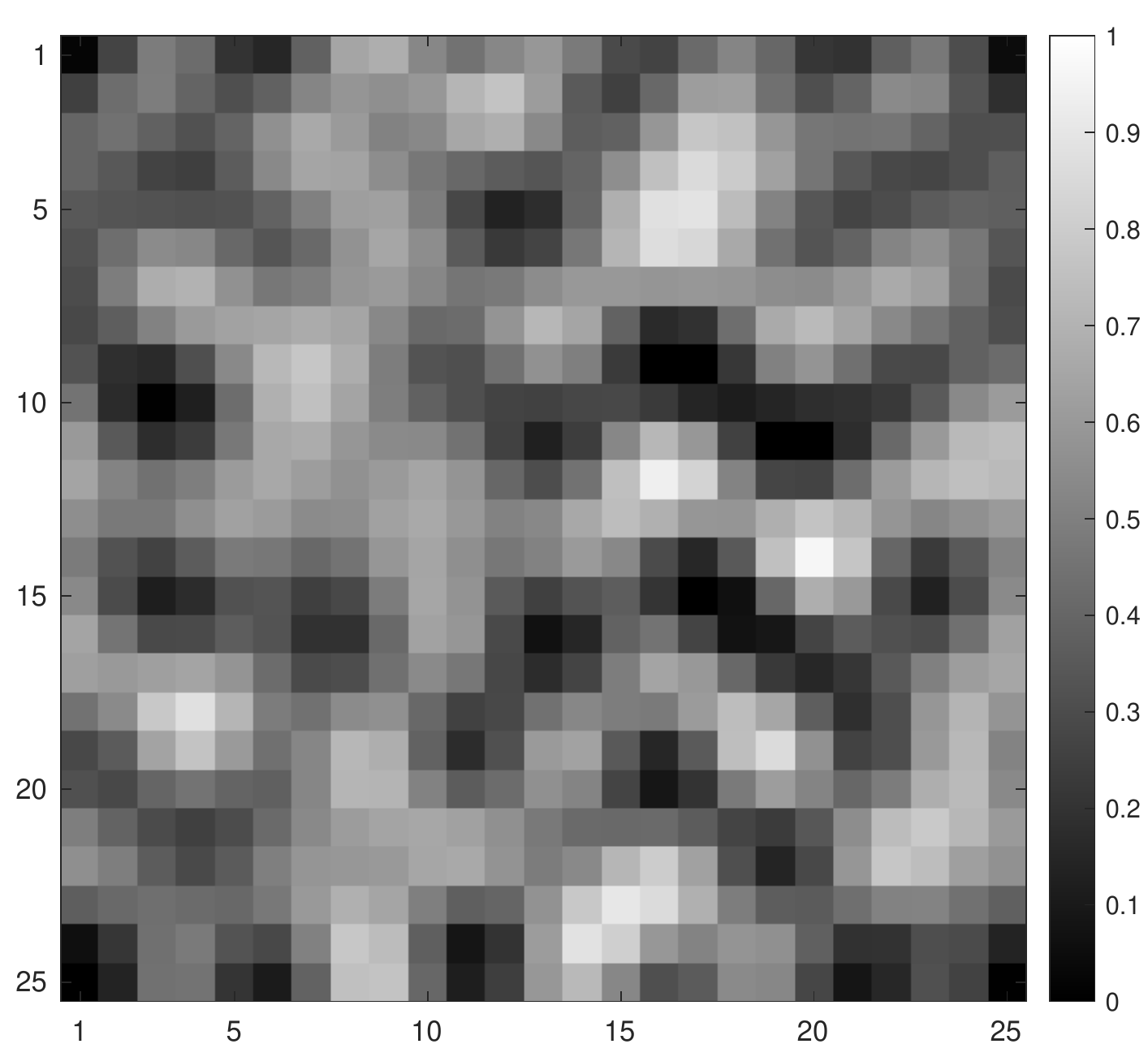}\  \includegraphics[width=4.13cm]{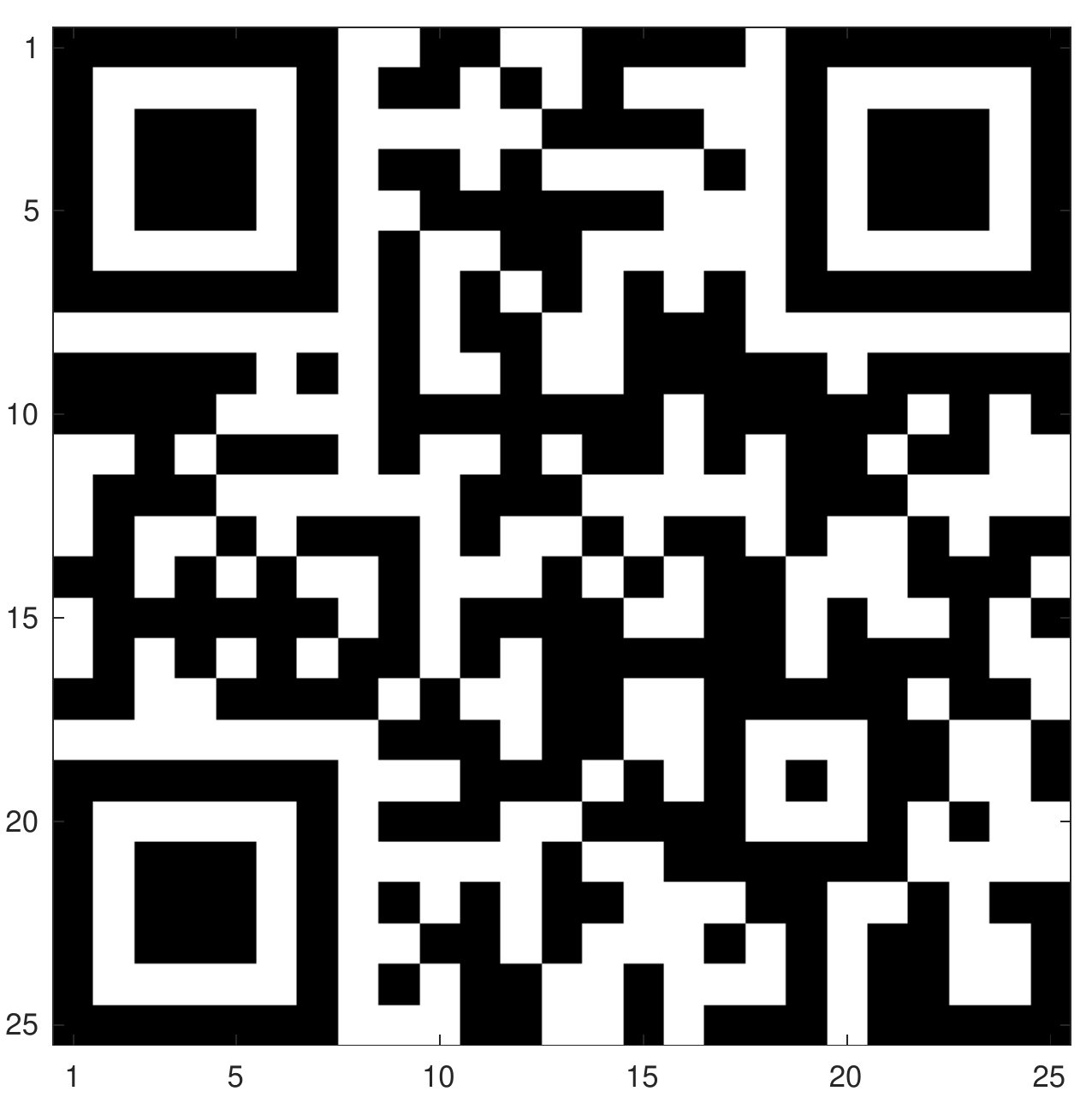}
\caption{\rm \label{fig:QR25} Left: band-limited reconstruction with $L=5$ of a Version 2 $25\times 25$ QR code with $S=287$ nonzero entries. Right: reconstruction by \cref{algo:case3} (exactly coincides with the model).}
\end{wrapfigure}

\section{Discussion}
\label{sec:disc}

We have shown that prior information that a matrix is binary allows one to reconstruct this matrix exactly from a limited set of DFT coefficients. Theoretically, for $N_1 \times N_2$ matrices with $N_1 \neq N_2$ both prime, only $4$ DFT coefficients are needed to guarantee uniqueness of this reconstruction regardless of the magnitudes of $N_1$ and $N_2$. For $N \times N$ matrices with a prime $N$, the number of required coefficients grows with $N$, but at a reasonable rate; the minimum band limit required for unique reconstruction is in this case $L=\lfloor\sqrt{N}\rfloor$. For square matrices of order $N=p^\alpha$, where $p$ is prime and $\alpha>1$ an integer, the minimum band limit is increased to $L=p^{\alpha-1}$. 

However, there exists a sizable gap between the theoretical guarantees of uniqueness and what is practical. The provided stability heuristics, which are supported by numerical examples, indicate that many digits of precision are needed in the data for reconstruction at the theoretical bounds. However, we have shown that it is possible to solve the problem even with a realistic amount of noise or imprecision in the DFT data by increasing the band limit past the theoretical bound while still not making all the coefficients available (in fact, far from that). This can also be understood by comparing the cases of square and non-square matrices with prime dimensions. In the former case, the band limit required to guarantee stability is significantly larger than in the latter case. However, we can always make a matrix square by making it larger (i.e., by adding rows or columns). Thus the theoretical results are counter-intuitive. For example, more DFT coefficients are required to recover uniquely a $29\times 29$ matrix than a $29\times 31$ matrix. However, with the account of stability, the apparent contradiction disappears. In order to reconstruct the two matrices {\em stably}, approximately the same number of DFT coefficients is needed.

In the numerical simulations, the algorithms combining integer linear programming (ILP) and Lenstra-Lenstra-Lovasz (LLL) lattice reduction were able to efficiently recover matrices as large as $29\times29$. In comparison, naive implementations of the ILP techniques fail for matrices as small as $5 \times 11$. However, even $29\times29$ matrices are on the smaller side of two-dimensional barcodes. It is therefore an open task to develop improved algorithms to handle larger binary matrix recovery in reasonable time. The current work mainly investigates recovery near the minimal band limit for uniqueness. It is worthwhile to investigate how these algorithms scale for larger matrices when $L$ is significantly larger than the minimum, while still not using all DFT coefficients.  Fast Fourier transform (FFT) and sparse FFT techniques are applicable when allowing for sampling of coefficients outside of the pass-band.  With even sparse sampling of a few high-frequency DFT coefficients could lead to scalable FFT based algorithms that have a smaller gap between theoretical results and practical reconstruction.

Additional constraints such as sparsity and connectivity can further increase computational feasibility for larger binary matrices, and allow for reconstruction with more significant noise. Sparse matrices with relatively small popcount $S$ can be considered straightforwardly by the algorithms developed here, and smallness of $S$ always entails greatly improved computational efficiency, with potential modifications.  For example, repeating the numerical experiment for \cref{algo:case2} from Section \ref{sec:sims2} for $N=29$ with smaller popcount $S=52$ resulted in about a 1 second reduction in average runtime (from 5 seconds to 4 seconds).  However, small modifications to the algorithm can increase computational efficiency further.  The overall size of the problem is significantly smaller for modest values of $S$.  In these cases, it is likely that fewer explicit directional sums are required to further reduce the overall problem to a manageable size.  For this same experiment with $S=52$, modifying the algorithm to only solve for four directional sums (row, column, and diagonals) resulted in an average run time of about 1.4 seconds, where all 100 randomly generated model matrices were successfully recovered.  Optimizing the algorithms for smaller values of $S$ is key ongoing work.  
 Connectivity is a conceptually different constraint, and its application can lead to improvements even for non-sparse matrices. 

Lastly, for applications to denoising corrupted QR codes, the algorithm can have improved computational efficiency by including additional prior information based on known QR code features. This includes fixed patterns, as well as masking that promotes disconnected images. QR codes also have built-in error correcting methods~\cite{tiwari_2016_1}. Combining this error correction with the proposed algorithms may yield efficient recovery with minimal available DFT coefficients and larger matrix sizes than $29 \times 29$.

\bibliographystyle{siamplain}
\bibliography{abbrev,Master,Book}

\end{document}

%% file: Shared.tex
\usepackage{amsmath,amssymb}
\usepackage{graphicx}
\usepackage{algorithmic}
\usepackage{wrapfig}
\hypersetup{breaklinks=true,urlcolor=blue}

\DeclareGraphicsExtensions{.pdf}
\graphicspath{{Graph/}}

\usepackage{enumitem}
\setlist[enumerate]{leftmargin=.5in}
\setlist[itemize]{leftmargin=.5in}

\headers{Inversion of band-limited discrete Fourier transforms}{H. W. Levinson, V. A. Markel, and N. Triantafillou}

\title{Inversion of band-limited discrete Fourier transforms of binary images: Uniqueness and algorithms}

\author{Howard W. Levinson\thanks{Department of Mathematics and Computer Science, Santa Clara University, Santa Clara, CA, USA (\email{hlevinson@scu.edu}).}
\and 
Vadim A. Markel\thanks{Department of Radiology, University of Pennsylvania, Philadelphia, PA, USA (\email{vmarkel@upenn.edu}).}
\and 
Nicholas Triantafillou\thanks{Center for Communications Research, Princeton, NJ, USA (\email{ngtriant@alum.mit.edu}).}}

\newcommand{\cu}{{\mathrm i}\, }

\allowdisplaybreaks

%% file: Paper.bbl
\begin{thebibliography}{10}

\bibitem{aardal_2000_1}
{\sc K.~Aardal, C.~A.~J. Hurkens, and A.~K. Lenstra}, {\em Solving a system of
  linear Diophantine equations with lower and upper bounds on the variables},
  \href{https://www.jstor.org/stable/3690477}{Math. Operations Res., 25 (2000), pp.~427--442}.

\bibitem{bailey_2012_1} {\sc J.~Bailey, M.~A. Iwen, and
  C.~V. Spencer}, {\em On the design of deterministic matrices for
  fast recovery of Fourier compressible functions},
  \href{https://doi.org/10.1137/110835864}{SIAM J. Matr. Analysis
    Appl., 33 (2012), pp.~263--289}.

\bibitem{beck_2009_1} {\sc A.~Beck and M.~Teboulle}, {\em A fast
  iterative shrinkage-thresholding algorithm for linear inverse
  problems}, \href{https://doi.org/10.1137/080716542}{SIAM
  J. Imag. Sci., 2 (2009), pp.~183--202}.

\bibitem{blumensath_2012_1} {\sc T.~Blumensath}, {\em Accelerated
  iterative hard thresholding},
  \href{https://doi.org/10.1016/j.sigpro.2011.09.017}{Sign.  Proc., 92
    (2012), pp.~752--756}.

\bibitem{blumensath_2013_1} {\sc T.~Blumensath}, {\em Compressed
  sensing with nonlinear oservations and related nonlinear
  optimization}, \href{https://doi.org/10.1109/TIT.2013.2245716}{IEEE
  Trans. Info. Theor., 59 (2013), pp.~3466--3474}.

\bibitem{blumensath_2009_1} {\sc T.~Blumensath and M.~E. Davies}, {\em
  Iterative hard thresholding for compressed sensing},
  \href{https://doi.org/10.1016/j.acha.2009.04.002}{Appl. Comp. Harm. Anal.,
    27 (2009), pp.~265--274}.

\bibitem{born_book_1999}
{\sc M.~Born and E.~Wolf}, {\em Principles of Optics}, Cambridge Univ. Press,
  1999.

\bibitem{borwein_2000_1} {\sc J.~M. Borwein and P.~Lison{\v{e}}k},
  {\em Applications of integer relation algorithms},
  \href{https://doi.org/10.1016/S0012-365X(99)00256-3}{Discr. Math.,
    217 (2000), pp.~65--82}.

\bibitem{brualdi_1980_1} {\sc R.~A. Brualdi}, {\em Matrices of zeros
  and ones with fixed row and column sum vectors},
  \href{https://doi.org/10.1016/0024-3795(80)90105-6}{Lin. Alg.
    Appl., 33 (1980), pp.~159--231}.

\bibitem{brualdi_1986_1} {\sc R.~A. Brualdi and E.~S. Solheid}, {\em
  On the spectral radius of complementary acyclic matrices of zeros
  and ones}, \href{https://doi.org/10.1137/0607030}{SIAM J. Alg. Disc.
  Meth., 7 (1986), pp.~265--272}.

\bibitem{candes_2006_1} {\sc E.~J. Cand{\`e}s, J.~Romberg, and
  T.~Tao}, {\em Robust uncertainty principles: {Exact} signal
  reconstruction from highly incomplete frequency information},
  \href{https://doi.org/10.1109/TIT.2005.862083}{IEEE
    Trans. Info. Theor., 52 (2006), pp.~489--509}.

\bibitem{conway_1976_1} {\sc J.~Conway and A.~Jones}, {\em
  Trigonometric diophantine equations (On vanishing sums of roots of
  unity)}, \href{http://eudml.org/doc/205475}{Acta Arithmetica, 30
  (1976), pp.~229--240}.

\bibitem{del_2002_1} {\sc A.~Del~Lungo, A.~Frosini, M.~Nivat, and
  L.~Vuillon}, {\em Discrete tomography: Reconstruction under
  periodicity constraints}, in
  \href{https://doi.org/10.1007/3-540-45465-9_5}{International
    Colloquium on Automata, Languages, and Programming, Springer,
    2002, pp.~38--56}.

\bibitem{dizenzo_1996_1} {\sc S.~Di~Zenzo, L.~Cinque, and
  S.~Levialdi}, {\em Run-based algorithms for binary image analysis
  and processing}, \href{https://doi.org/10.1109/34.476016}{IEEE
  Trans. Pattern. Anal. Mach.  Intel., 18 (1996), pp.~83--89}.

\bibitem{donoho_1989_1} {\sc D.~L. Donoho and P.~B. Stark}, {\em
  Uncertainty principles and signal recovery},
  \href{https://doi.org/10.1137/0149053}{SIAM J. Appl. Math., 49
    (1989), pp.~906--931}.

\bibitem{elad_2002_1} {\sc M.~Elad and A.~M. Bruckstein}, {\em A
  generalized uncertainty principle and sparse representation in pairs
  of bases}, \href{https://doi.org/10.1109/TIT.2002.801410}{IEEE
  Trans. Info. Theor., 48 (2002), pp.~2558--2567}.

\bibitem{ferguson_1999_1} {\sc H.~Ferguson, D.~Bailey, and S.~Arno},
  {\em Analysis of PSLQ, an integer relation finding algorithm},
  \href{http://dx.doi.org/10.1090/S0025-5718-99-00995-3}{Math. Comp.,
    68 (1999), pp.~351--369}.

\bibitem{fulkerson_1960_1} {\sc D.~R. Fulkerson}, {\em Zero-one
  matrices with zero trace},
  \href{https://projecteuclid.org/journals/pacific-journal-of-mathematics/volume-10/issue-3/Zero-one-matrices-with-zero-trace/pjm/1103038231.full?tab=ArticleLink}{Pacific
    J.  Math., 10 (1960), pp.~831--836}.

\bibitem{gao_2007_1} {\sc J.~Z. Gao, L.~Prakash, and R.~Jagatesan},
  {\em Understanding 2D-barcode technology and applications in
    M-commerce-design and implementation of a 2D barcode processing
    solution}, in \href{https://doi.org/10.1109/COMPSAC.2007.229}{31st
    Ann. Int. Computer Software and Applications Conference, vol.~2,
    IEEE, 2007, pp.~49--56}.

\bibitem{gardner_1997_1} {\sc R.~Gardner and P.~Gritzmann}, {\em
  Discrete tomography: Determination of finite sets by X-rays},
  \href{https://doi.org/10.1016/j.dam.2012.09.010}{Trans. Am. Math. Soc.,
    349 (1997), pp.~2271--2295}.

\bibitem{gardner_1999_1} {\sc R.~J. Gardner, P.~Gritzmann, and
  D.~Prangenberg}, {\em On the computational complexity of
  reconstructing lattice sets from their X-rays},
  \href{https://doi.org/10.1016/S0012-365X(98)00347-1}{Discr. Math.,
    202 (1999), pp.~45--71}.

\bibitem{gilbert_2014_1} {\sc A.~C. Gilbert, P.~Indyk, M.~Iwen, and
  L.~Schmidt}, {\em Recent developments in the sparse Fourier
  transform: A compressed Fourier transform for big data},
  \href{https://doi.org/10.1109/MSP.2014.2329131}{IEEE Signal
    Proc. Mag., 31 (2014), pp.~91--100}.

\bibitem{greenwood_1955_1} {\sc J.~A. Greenwood and D.~Durand}, {\em
  The distribution of length and components of the sum of $n$ random
  unit vectors},
  \href{https://www.jstor.org/stable/2236878}{Ann. Math. Stat., 26
    (1955), pp.~233--246}.

\bibitem{habegger_2018_1} {\sc P.~Habegger}, {\em The norm of Gaussian
  periods}, \href{https://doi.org/10.1093/qmath/hax028}{Quart. J.
  Math., 69 (2018), pp.~153--182}.

\bibitem{hajdu_2001_1} {\sc L.~Hajdu and R.~Tijdeman}, {\em Algebraic
  aspects of discrete tomography},
  \href{https://doi.org/10.1515/crll.2001.037https://doi.org/10.1515/crll.2001.037}{J. Reine
    Angew. Math., 534 (2001), pp.~119--128}.

\bibitem{hartshorne_book_2013} {\sc R.~Hartshorne}, {\em Geometry:
  Euclid and Beyond}, Undergraduate Texts in Mathematics, Springer,
  2013.

\bibitem{hastad_1989_1} {\sc J.~Hastad, B.~Just, J.~C. Lagarias, and
  C.-P. Schnorr}, {\em Polynomial time algorithms for finding integer
  relations among real numbers},
  \href{https://doi.org/10.1007/3-540-16078-7_69}{SIAM J.  Computing,
    18 (1989), pp.~859--881}.

\bibitem{herman_2003_1} {\sc G.~T. Herman and A.~Kuba}, {\em Discrete
  tomography in medical imaging},
  \href{https://doi.org/10.1109/JPROC.2003.817871}{Proc. IEEE, 91
    (2003), pp.~1612--1626}.

\bibitem{herman_book_2012}
{\sc G.~T. Herman and A.~Kuba}, {\em Discrete Tomography: Foundations,
  Algorithms, and Applications}, Springer, 2012.

\bibitem{hu_2014_1} {\sc W.~Hu, G.~Cheung, A.~Ortega, and O.~C. Au},
  {\em Multiresolution graph Fourier transform for compression of
    piecewise smooth images},
  \href{https://doi.org/10.1109/TIP.2014.2378055}{IEEE Trans.
    Imag. Proc., 24 (2014), pp.~419--433}.

\bibitem{jiji_2006_1} {\sc C.~V. Jiji, P.~Neethu, and S.~Chaudhuri},
  {\em Alias-free interpolation}, in
  \href{https://doi.org/10.1007/11744085_20}{Computer Vision - ECCV
    2006, A.~Leonardis, H.~Bischof, and A.~Pinz, eds., Springer, 2006,
    pp.~255--266}.

\bibitem{karp_1972_1} {\sc R.~M. Karp}, {\em Reducibility among
  combinatorial problems}, in
  \href{https://doi.org/10.1007/978-1-4684-2001-2_9}{Complexity of
    Computer Computations, Springer, 1972, pp.~85--103}.

\bibitem{lagendijk_1990_1} {\sc R.~L. Lagendijk and J.~Biemond}, {\em
  Iterative Identification and Restoration of Images}, Springer, 1990.

\bibitem{lam_2000_1} {\sc T.~Y. Lam and K.~H. Leung}, {\em On
  vanishing sums of roots of unity},
  \href{https://doi.org/10.1006/jabr.1999.8089}{J. Algebra, 224
    (2000), pp.~91--109}.

\bibitem{lenstra_1982_1} {\sc A.~K. Lenstra, H.~W. Lenstra, and
  L.~Lov{\'a}sz}, {\em Factoring polynomials with rational
  coefficients},
  \href{https://doi.org/10.1007/BF01457454}{Mathematische Annalen, 261
    (1982), pp.~515--534}.

\bibitem{lenstra_1978_1} {\sc H.~W. Lenstra}, {\em Vanishing sums of
  roots of unity}, in
  \href{https://www.math.leidenuniv.nl/~hwl/PUBLICATIONS/1979d/art.pdf}{Proc. Bicentennial
    Congress Wiskundig Genootschap, Part II, Vrije Univ. Amsterdam,
    1978, pp.~249--268}.

\bibitem{lenstra_1983_1} {\sc H.~W. Lenstra}, {\em Integer programming
  with a fixed number of variables},
  \href{https://www.jstor.org/stable/3689168}{Math. Oper. Res., 8
    (1983), pp.~538--548}.

\bibitem{levinson_2021_1} {\sc H.~W. Levinson and V.~A. Markel}, {\em
  Binary discrete Fourier transform and its inversion},
  \href{https://doi.org/10.1109/TSP.2021.3088215}{IEEE
    Trans. Sign. Proc., 69 (2021), pp.~3484--3499}.

\bibitem{liu_2014_1} {\sc X.~Liu, D.~Zhai, D.~Zhao, G.~Zhai, and
  W.~Gao}, {\em Progressive image denoising through hybrid graph
  laplacian regularization: A unified framework},
  \href{https://doi.org/10.1109/TIP.2014.2303638}{IEEE
    Trans. Imag. Proc., 23 (2014), pp.~1491--1503}.

\bibitem{marchand_2002_1} {\sc H.~Marchand, A.~Martin, R.~Weismantel,
  and L.~Wolsey}, {\em Cutting planes in integer and mixed integer
  programming},
  \href{https://doi.org/10.1016/S0166-218X(01)00348-1}{Discr. Appl. Math.,
    123 (2002), pp.~397--446}.

\bibitem{marchand_book_1999} {\sc S.~Marchand-Maillet and
  Y.~M. Sharaiha}, {\em Binary Digital Image Processing: A Discrete
  Approach}, Elsevier, 1999.

\bibitem{maznev_2017_1} {\sc A.~A. Maznev and O.~B. Wright}, {\em
  Upholding the diffraction limit in the focusing of light and sound},
  \href{https://doi.org/10.1016/j.wavemoti.2016.09.012}{Wave Motion,
    68 (2017), pp.~182--189}.

\bibitem{moitra_2015_1} {\sc A.~Moitra}, {\em Super-resolution,
  extremal functions and the condition number of Vandermonde
  matrices}, in
  \href{https://doi.org/10.1145/2746539.2746561}{Proceedings of the
    Forty-Seventh Annual ACM Symposium on Theory of Computing, ACM,
    2015, pp.~821--830}.

\bibitem{myerson_1986_1} {\sc G.~Myerson}, {\em How small can a sum of
  roots of unity be?}, \href{https://doi.org/10.2307/2323469}{The
  American Mathematical Monthly, 93 (1986), pp.~457--459}.

\bibitem{nasrollahi_2014_1} {\sc K.~Nasrollahi and T.~B. Moeslund},
  {\em Super-resolution: A comprehensive survey},
  \href{https://doi.org/10.1007/s00138-014-0623-4}{Machine Vision and
    Applications, 25 (2014), pp.~1423--1468}.

\bibitem{nguyen_2004_1} {\sc P.~Q. Nguyen and D.~Stehl{\'e}}, {\em
  Low-dimensional lattice basis reduction revisited}, in
  \href{https://doi.org/10.1145/1597036.1597050}{International
    Algorithmic Number Theory Symposium, Springer, 2004,
    pp.~338--357}.

\bibitem{pei_2022_1} {\sc S.-C. Pei and K.-W. Chang}, {\em Binary
  signal perfect recovery from partial DFT coefficients},
  \href{https://doi.org/10.1109/TSP.2022.3190615}{IEEE
    Trans. Sign. Proc., 70 (2022), pp.~3848--3861}.

\bibitem{plonka_2021_1} {\sc G.~Plonka and T.~von Wulffen}, {\em
  Deterministic sparse sublinear FFT with improved numerical
  stability},
  \href{https://doi.org/10.1007/s00025-020-01330-0}{Results in
    Mathematics, 76 (2021), p.~53}.

\bibitem{plonka_2018_1} {\sc G.~Plonka, K.~Wannenwetsch, A.~Cuyt, and
  W.-S. Lee}, {\em Deterministic sparse FFT for M-sparse vectors},
  \href{https://doi.org/10.1007/s11075-017-0370-5}{Numerical
    Algorithms, 78 (2018), pp.~133--159}.

\bibitem{rajan_2001_1} {\sc D.~Rajan and S.~Chaudhuri}, {\em
  Generalized interpolation and its application in super-resolution
  imaging}, \href{https://doi.org/10.1016/S0262-8856(01)00055-5}{Image
  and Vision Computing, 19 (2001), pp.~957--969}.

\bibitem{rauhut_2007_1} {\sc H.~Rauhut}, {\em Random sampling of
  sparse trigonometric polynomials},
  \href{https://doi.org/10.1016/j.acha.2006.05.002}{Appl. Comp. Harm. Anal.,
    22 (2007), pp.~16--42}.

\bibitem{ren_2002_1} {\sc M.~Ren, J.~Yang, and H.~Sun}, {\em Tracing
  boundary contours in a binary image},
  \href{https://doi.org/10.1016/S0262-8856(01)00091-9}{Image Vision
    Comp., 20 (2002), pp.~125--131}.

\bibitem{romano_2017_1} {\sc Y.~Romano, M.~Elad, and M.~Peyman}, {\em
  The little engine that could: Regularization by denoising (RED)},
  \href{https://doi.org/10.1137/16M1102884}{SIAM J. Imag. Sci., 10
    (2017), pp.~1804--1844}.

\bibitem{rudelson_2006_1} {\sc M.~Rudelson and R.~Vershynin}, {\em
  Sparse reconstruction by convex relaxation: Fourier and Gaussian
  measurements}, in
  \href{https://doi.org/10.1109/CISS.2006.286463}{2006 40th Annual
    Conference on Information Sciences and Systems, IEEE, 2006,
    pp.~207--212}.

\bibitem{rudin_1992_1} {\sc L.~I. Rudin, S.~Osher, and E.~Fatemi},
  {\em Nonlinear total variation based noise removal algorithms},
  \href{https://doi.org/10.1016/0167-2789(92)90242-F}{Physica D, 60
    (1992), pp.~259--268}.

\bibitem{ryser_1957_1} {\sc H.~J. Ryser}, {\em Combinatorial
  properties of matrices of zeros and ones},
  \href{https://doi.org/10.4153/CJM-1957-044-3}{Canadian J. Math., 9
    (1957), pp.~371--377}.

\bibitem{ryser_1960_1} {\sc H.~J. Ryser}, {\em Matrices of zeros and
  ones},
  \href{https://doi.org/10.1090/S0002-9904-1960-10494-6}{Bull. Am. Math. Sci.,
    66 (1960), pp.~442--464}.

\bibitem{schrijver_1980_1} {\sc A.~Schrijver}, {\em On cutting
  planes},
  \href{https://doi.org/10.1016/S0167-5060(08)70085-2}{Ann. Discr. Math.,
    9 (1980), pp.~291--296}.

\bibitem{shoup_book_2009}
{\sc V.~Shoup}, {\em A Computational Introduction to Number Theory and
  Algebra}, Cambridge Univ. Press, 2009.

\bibitem{snijders_1991_1} {\sc T.~A.~B. Snijders}, {\em Enumeration
  and simulation methods for 0--1 matrices with given marginals},
  \href{https://doi.org/10.1007/BF02294482}{Psychometrika, 56 (1991),
    pp.~397--417}.

\bibitem{tao_2005_1}
{\sc T.~Tao}, {\em An uncertainty principle for cyclic groups of prime order},
  \href{https://dx.doi.org/10.4310/MRL.2005.v12.n1.a11}{Math. Res. Lett., 12 (2005), pp.~121--127}.

\bibitem{tiwari_2016_1} {\sc S.~Tiwari}, {\em An introduction to QR
  code technology}, in
  \href{https://doi.org/10.1109/ICIT.2016.021}{International
    Conference on Information Technology (ICIT), IEEE, 2016,
    pp.~39--44}.

\bibitem{tropp_2008_1} {\sc J.~A. Tropp}, {\em On the linear
  independence of spikes and sines},
  \href{https://doi.org/10.1007/s00041-008-9042-0}{J. Fourier
    Anal. Appl., 14 (2008), p.~838}.

\bibitem{tropp_2007_1} {\sc J.~A. Tropp and A.~C. Gilbert}, {\em
  Signal recovery from random measurements via orthogonal matching
  pursuit}, \href{https://doi.org/10.1109/TIT.2007.909108}{IEEE
  Trans. Info. Theor., 53 (2007), pp.~4655--4666}.

\bibitem{wang_1998_1} {\sc B.-Y. Wang and F.~Zhang}, {\em On the
  precise number of (0,1)-matrices in A(R,S)},
  \href{https://doi.org/10.1016/S0012-365X(97)00197-0}{Discr. Math.,
    187 (1998), pp.~211--220}.

\bibitem{yagle_1998_1} {\sc A.~E. Yagle}, {\em An algebraic solution
  to the 3-D discrete tomography problem}, in
  \href{https://doi.org/10.1109/ICIP.1998.723627}{Proc. 1998
    Int. Conf. Image Process., vol.~2, IEEE, 1998, pp.~714--717}.

\bibitem{yagle_2001_1} {\sc A.~E. Yagle}, {\em A convergent composite
  mapping Fourier domain iterative algorithm for 3-D discrete
  tomography},
  \href{https://doi.org/10.1016/S0024-3795(01)00458-X}{Lin. Alg. Appl.,
    339 (2001)}, pp.~91--109.

\end{thebibliography}
